\documentclass[11pt]{article}
\usepackage{amscd, amsmath, amssymb, amsthm}
\usepackage[all,cmtip]{xy}
\usepackage[pagebackref]{hyperref}

\title{Cohomological invariants in positive characteristic }
\author{Burt Totaro}
\date{  }

\def\Z{\text{\bf Z}}

\def\P{\text{\bf P}}
\def\F{\text{\bf F}}

\DeclareMathOperator{\tame}{tame}
\DeclareMathOperator{\nr}{nr}
\DeclareMathOperator{\et}{et}
\DeclareMathOperator{\Spec}{Spec}
\def\arrow{\rightarrow}
\def\inj{\hookrightarrow}

\DeclareMathOperator{\Inv}{Inv}
\DeclareMathOperator{\Hom}{Hom}
\DeclareMathOperator{\Gal}{Gal}
\DeclareMathOperator{\Pic}{Pic}
\DeclareMathOperator{\NormInv}{NormInv}
\DeclareMathOperator{\Zar}{Zar}
\DeclareMathOperator{\bal}{bal}
\DeclareMathOperator{\Br}{Br}
\DeclareMathOperator{\ch}{char}
\DeclareMathOperator{\disc}{disc}
\DeclareMathOperator{\clif}{clif}
\DeclareMathOperator{\Spin}{Spin}

\DeclareMathOperator{\tr}{tr}
\def\o{\overline}
\def\m{\mathfrak{m}}

\DeclareMathOperator{\bi}{b}
\def\ll{\langle\langle}
\def\etale{\'etale }

\begin{document}
\maketitle
\newtheorem{theorem}{Theorem}[section]
\newtheorem{proposition}[theorem]{Proposition}
\newtheorem{corollary}[theorem]{Corollary}
\newtheorem{lemma}[theorem]{Lemma}

\theoremstyle{definition}
\newtheorem{definition}[theorem]{Definition}
\newtheorem{example}[theorem]{Example}

\theoremstyle{remark}
\newtheorem{remark}[theorem]{Remark}

\'Etale cohomology works especially well with $\Z/l$ coefficients
such that $l$ is invertible in the base field.
In the 1980s, however, Kato used differential forms
to define groups $H^i_{\et}(k,\Z/m(j))$ for a field $k$
and any positive integer $m$,
even when $m$ is not invertible in $k$
\cite[p.~219]{Katocomplete}. Kato's groups behave surprisingly well.
For example, we have $H^1_{\et}(k,\Z/m(0))\cong H^1_{\et}(k,\Z/m)$,
the group classifying cyclic $\Z/m$-extensions of $k$, and
$H^2_{\et}(k,\Z/m(1))\cong \Br(k)[m]$, the $m$-torsion subgroup
of the Brauer group, whether $m$ is invertible in $k$
or not.

Nowadays, there is an ``explanation'' for Kato's groups: Voevodsky's
\etale motivic cohomology groups $H^i_{\et}(X,A(j))$ of a scheme $X$
over a field $k$ are defined
for any abelian group $A$. They agree with the familiar \etale
cohomology with coefficients in $\mu_m^{\otimes j}$ when $A$ is $\Z/m$ with $m$
invertible in $k$ and $j\geq 0$,
and they agree with Kato's groups when $X=\Spec(k)$ and $A=\Z/m$
for any $m$
\cite[Theorem 10.2]{MVW}, \cite{GL}.

In this paper, we make some new calculations
of mod $p$ \etale motivic cohomology in characteristic $p$.
In particular, we compute
the group of cohomological invariants (in Serre's sense)
for some important affine group schemes, such as the symmetric groups
(Theorem \ref{symm2}),
the finite group schemes $(\mu_p)^a\times (\Z/p)^b$ (Theorem \ref{abgp}),
and the orthogonal groups $O(n)$ and $SO(n)$
(Theorems \ref{evenorthog}, \ref{oddorthog},
\ref{oddplus}, \ref{oplus}). These calculations were
done in \cite[Chapters VI and VII]{GMS}
for $\Z/l$ coefficients with $l\neq p$,
and we carry out the case $l=p$. For the orthogonal groups,
the interesting new case is where these groups are considered
over a field of characteristic 2. In that case, our
calculation amounts to determining the group of cohomological
invariants for quadratic forms in characteristic 2.

One outcome of the calculations is that there are often fewer mod $p$
cohomological invariants when the base field has characteristic $p$.
For example, a basis for the mod 2 cohomological invariants for the orthogonal
group $O(n)$ in characteristic not 2 is given by the Stiefel-Whitney
classes $1=w_0, w_1,\ldots,w_n$, whereas in characteristic 2 there are only
analogs of $w_1$ and $w_2$, the discriminant (or Arf invariant)
and the Clifford invariant.
In particular, cohomological invariants are not enough to give the lower bounds
for the essential dimension of $O(n)$ and $SO(n)$ in characteristic 2 proved
by Babic and Chernousov \cite{BC}. The cohomological invariants
of the spin groups $\Spin(n)$ in characteristic 2
(as in other characteristics)
are not known, but for small $n$ there are enough
invariants to give optimal lower bounds on the essential dimension
\cite{Totarospin}.

We also determine all operations on the mod $p$ \etale motivic cohomology
of fields
(section \ref{operationsect}),
extending Vial's computation of the operations on the mod $p$
Milnor $K$-theory of fields \cite{Vial}.

As far as I know, this paper gives the first calculations
of all mod $p$ cohomological invariants for a given affine
group scheme in characteristic $p$. We use a geometric description
of such invariants by Blinstein and Merkurjev
(Theorem \ref{bm} below),
but the only full calculations seem to be in low degrees.
In particular, Blinstein and Merkurjev described
the cohomological invariants in degrees at most 3 for tori
\cite{BM}. Merkurjev determined all degree 3 invariants
for simply connected semisimple groups; they are generated
by the Rost invariant in the mod $p$ case, as in
the mod $l$ case \cite[Part 2, Theorem 9.11]{GMS}.

A key difference between \etale motivic cohomology in the mod $p$
case and the mod $l$ case is that mod $p$ \etale motivic cohomology
of schemes
is not $A^1$-homotopy invariant. (For example, for $k$ algebraically
closed of characteristic $p$, $H^1_{\et}(k,\Z/p)$ is zero, while
$H^1_{\et}(A^1_k,\Z/p)$ is not zero: there are many nontrivial
\etale $\Z/p$-coverings of the affine line.) This failure is related
to the phenomenon of wild ramification (section \ref{ramsection}),
which does not occur
in the mod $l$ case. One goal of this paper
is to show that, although the lack of $A^1$-homotopy invariance
means that some familiar arguments no longer apply,
mod $p$ \etale motivic cohomology is still a useful
and computable theory. A crucial ingredient of the proofs
is an analysis of tame and wild ramification
for classes in \etale motivic cohomology, extending work
of Izhboldin (Theorem \ref{filtration}).

This work was supported by NSF grant DMS-1701237.

\tableofcontents

\section{Background on \etale motivic cohomology}
\label{background}

Building on earlier work of Bloch and Kato,
Geisser and Levine proved the relation between Voevodsky's
\etale motivic cohomology and Kato's invariants of fields based
on differential forms. Namely, let $k$ be a field of characteristic $p>0$
which is perfect, meaning that every element of $k$ is a $p$th power,
and let $X$ be a smooth scheme over $k$.
For $j\geq 0$,
let $\Omega^j_{\log}$ be the subsheaf of $\Omega^j_X$ generated locally
by logarithmic differentials $df_1/f_1\wedge\cdots
\wedge df_j/f_j$ for units $f_1,\ldots,f_j$. (This is a sheaf
of $\F_p$-vector spaces, not of $O_X$-modules.) More generally,
for $r>0$, let $W_r\Omega^j_{\log}$ be the analogous subsheaf
of logarithmic de Rham-Witt differentials \cite{Illusie}.
Then Voevodsky's object $\Z/p^r(j)$ 
in the derived category of Zariski (or \etale) sheaves on $X$ is isomorphic
to the shift $W_r\Omega^j_{\log}[-j]$ \cite[Proposition 3.1,
Theorem 8.3]{GL}. As a result, \etale motivic cohomology,
meaning the \etale cohomology of $X$ with coefficients in $\Z/p^r(j)$,
can be rewritten in terms of differential forms:
$$H^i_{\et}(X,\Z/p^r(j))\cong H^{i-j}_{\et}(X,W_r\Omega^j_{\log}).$$

This has consequences for any field $k$ of characteristic $p$,
not necessarily perfect. Indeed, such a field has \etale $p$-cohomological
dimension at most 1.  As a result, $H^i_{\et}(k,\Z/p^r(j))$ is zero
except when $i$ is $j$ or $j+1$. When $i=j$, Bloch and Kato
identified this group with the Milnor $K$-group
$K_j^M(k)/p^r$, or also with the group $W_r\Omega^j_{\log,k}$
\cite[Corollary 2.8]{BK}.
There are several ways to describe the remaining mod $p^r$ \etale
motivic cohomology groups of a field, when $i=j+1$; we concentrate
on the case $r=1$. 

Write $H^{i,j}(k)=H^i_{\et}(k,\Z/p(j))$.
One description of these groups is in terms of Galois cohomology.
For a field $k$ of characteristic $p>0$, let $k_s$ be a separable
closure of $k$. Let $\Omega^j_k$ be the group
of (absolute) differential forms on $k$, which can be viewed
as $\Omega^j_{k/\Z}$ or $\Omega^j_{k/\F_p}$.
Write $\Omega^j_{\log,k}$ for the subgroup
of $\Omega^j_k$ generated by elements $(da_1/a_1)\wedge\cdots
\wedge (da_j/a_j)$ with $a_1,\ldots,a_j$ in $k^*$. Then
$$H^{i,j}(k)\cong \begin{cases}
\Omega^j_{\log,k}\cong H^0_{\Gal}(k,\Omega^j_{\log,k_s})
&\text{if }i=j\\
H^1_{\Gal}(k,\Omega^j_{\log,k_s})&\text{if }i=j+1\\
0&\text{otherwise.}
\end{cases}$$
The Galois group $\Gal(k_s/k)$ of a field $k$ of characteristic $p>0$
has $p$-cohomological dimension at most 1 \cite[section II.2.2]{SerreGalois},
which explains why only $H^0$ and $H^1$ occur here.

For another description of these groups (Kato's original definition
\cite{Katoquadratic}),
define a group homomorphism
$\mathcal{P}\colon \Omega^j_k\arrow \Omega^j_k/d\Omega^{j-1}_k$
by
$$\mathcal{P}(a(db_1/b_1)\wedge\cdots\wedge(db_j/b_j))=
(a^p-a)(db_1/b_1)\wedge\cdots\wedge(db_j/b_j).$$
Then $H^{j+1,j}(k)$ is isomorphic to
the cokernel of $\mathcal{P}$
\cite[Corollary 6.5]{Izhboldintorsion}.
In fact, there is an exact sequence:
$$0\arrow H^{j,j}(k)\arrow \Omega^j_k\xrightarrow[\mathcal{P}]{}
\Omega^j_k/d\Omega^{j-1}_k\arrow H^{j+1,j}(k)\arrow 0.$$

Let $k$ be a field of characteristic $p>0$. For an element $a$
of $k$, write $[a]$ for the class of $a$ in $H^{1,0}(k)=k/\mathcal{P}(k)$,
where $\mathcal{P}(a)=a^p-a$, as above.
For $b_1,\ldots,b_j$ in $k^*$, the {\it symbol }$\{b_1,\ldots,b_j\}$
in $H^{j,j}(k)$ means the class of the differential form
$(db_1/b_1)\wedge\cdots\wedge (db_j/b_j)$; this agrees
with the standard notation in Milnor $K$-theory, via the isomorphism
$H^{j,j}(k)\cong K_j^M(k)/p$.
Finally, for $a\in k$ and $b_1,\ldots,b_j\in k^*$,
the symbol 
$$[a,b_1,\ldots,b_j\}\in H^{j+1,j}(k)$$
means the class
of the differential form $a(db_1/b_1)\wedge\cdots\wedge(db_j/b_j)$.
Both groups $H^{j,j}(k)$ and $H^{j+1,j}(k)$ are generated by symbols,
by the descriptions above.

For a scheme $X$ of characteristic $p$,
\etale motivic cohomology with $\Z/l(j)$ coefficients for $l\neq p$
and $j\geq 0$ can be
identified with \etale cohomology with the familiar coefficients
$\mu_l^{\otimes j}$. (For $X$ smooth over $k$, which is the only
case we will need, this is \cite[Theorem 10.2]{MVW}.)
In particular, it follows that \etale motivic cohomology with
$\Z/l(j)$ coefficients with $l\neq p$
is $A^1$-homotopy invariant, by one of Grothendieck's
fundamental results \cite[Corollary VI.4.20]{Milne}.
By contrast, mod $p$ \etale motivic cohomology
is not $A^1$-homotopy invariant. For a simple example, look
at $H^{1,0}(X)\cong H^1_{\et}(X,\Z/p)$. We have the
Artin-Schreier exact sequence of \etale sheaves:
$$0\arrow \Z/p\arrow O_X\xrightarrow[\mathcal{P}]{} O_X\arrow 0,$$
where $\mathcal{P}(a)=a^p-a$. For $X$ affine, it follows that
we have an exact sequence
$$O(X) \xrightarrow[\mathcal{P}]{} O(X)\arrow H^1_{\et}(X,\Z/p)\arrow 0.$$
For example, if $k$ is an algebraically closed field, then
$H^1_{\et}(k,\Z/p)=0$, whereas one checks from this exact sequence
that $H^1_{\et}(A^1_k,\Z/p)$ is isomorphic to a countably infinite direct
sum of copies of $k$.

Let $G$ be an affine group scheme of finite type over a field $k$.
This determines a functor from fields over $k$
to sets by $F\mapsto H^1(F,G)$, the set of isomorphism classes of $G$-torsors
over $F$. (Here $G$-torsors are defined in the most general sense,
using the fppf topology; for $G$ smooth over $k$, this is the same
as $G$-torsors in the \etale topology \cite[Remark III.4.8]{Milne}.)
The abelian group of {\it cohomological invariants }of $G$ with values in
$H^i_{\et,\Z/m(j)}$, written $\Inv^i_k(G,\Z/m(j))$,
means the set of natural transformations
from $H^1(F,G)$ to $H^i_{\et}(F,\Z/m(j))$, on the category of fields
$F$ over $k$. When the positive integer $m$ is invertible in $k$,
the group of cohomological invariants was computed for several
important groups $G$ in \cite[Chapters VI and VII]{GMS}:
the symmetric groups, elementary abelian groups,
and the orthogonal groups. In this paper, we will make the analogous
mod $p$ calculations when $p$ is the characteristic of $k$.

A cohomological invariant for a group scheme $G$ over $k$
is {\it normalized }if it is equal to zero on the trivial $G$-torsor.
It is immediate that the group of invariants for $G$ splits as the direct sum
of the ``constant'' invariants and the normalized invariants:
$$\Inv^i_k(G,\Z/m(j))\cong H^{i}(k,\Z/m(j))\oplus \NormInv^{i}_k(G,
\Z/m(j)).$$

Some insight into the group of cohomological invariants
is provided by the existence of a {\it versal torsor}.
Let $G$ be
an affine group scheme over an infinite field $k$, and let $V$
be a $k$-vector space on which $G$ acts by affine transformations.
Suppose that $G$ acts
freely on a nonempty Zariski open subset $U$ of $V$, with a quotient
scheme $U/G$. Then every $G$-torsor over an extension field of $k$
is pulled back from the $G$-torsor $U\arrow U/G$
\cite[section I.5]{GMS}. As a result, we have an injection
$$\Inv^i_k(G,\Z/m(j))\inj H^i(k(U/G), \Z/m(j)).$$
Also, cohomological invariants always give cohomology classes
on $k(U/G)$ that are unramified along all divisors in $U/G$.

For $m$ invertible in $k$, this injection is in fact
an isomorphism to the group $H^0(U/G,H^i)$ of unramified classes,
under the mild extra assumption
that $V-U$ has codimension at least 2 in $V$ \cite[Part 1,
Appendix C]{GMS}. However, that argument relies on
$A^1$-homotopy invariance. For $p=\ch(k)$, where
$A^1$-homotopy invariance fails, one cannot
expect to identify the mod $p$ cohomological invariants
of $G$
with the unramified cohomology of a quotient variety $U/G$;
consider the case of the trivial group $G$ and
vector spaces $U=V$ of various dimensions. However, Blinstein
and Merkurjev provided a substitute: for any
positive integer $m$, the group
of cohomological invariants for $G$ need not be the whole
group $H^i_{\nr}(k(U/G),\Z/m(j))$, but it is always the subgroup
of {\it balanced }elements in $H^i(k(U/G),\Z/m(j))$, meaning the elements
whose pullbacks via the two projections $(U\times U)/G\arrow
U/G$ are equal. Balanced elements
are always unramified, and so the group of cohomological
invariants can also be described as the subgroup of balanced
elements in unramified cohomology \cite[Theorem A]{BM}:

\begin{theorem}
\label{bm}
Let $G$ be an affine group scheme of finite type over an infinite
field $k$.
Let $U$ be a smooth $k$-variety with a free $G$-action
such that there is a quotient scheme $U/G$.
Suppose that $U$ is $G$-equivariantly birational to an affine space
over $k$ on which $G$ acts by affine transformations. Then
\begin{align*}
\Inv^i_k(G,\Z/m(j)) &\cong H^i(k(U/G),\Z/m(j))_{\bal}\\
&\cong H^0_{\Zar}(U/G,H^i_{\Z/m(j)})_{\bal}.
\end{align*}
\end{theorem}

Our calculation of the cohomological invariants of the group scheme
$\mu_p$
(Proposition \ref{mup}), on which the rest of the paper depends,
relies on Theorem \ref{bm}.

\section{Ramification and residues}
\label{ramsection}

In this section, building on the work of Izhboldin, we describe
\etale motivic cohomology for a field with a discrete valuation.
In particular, there are notions of tame and wild ramification
for cohomology classes,
and a residue homomorphism. The quotient of \etale motivic cohomology
by the unramified subgroup can be described very explicitly
(Theorem \ref{filtration}).
Finally, we state Izhboldin's calculation of the \etale motivic
cohomology of a rational function field (Theorem \ref{k(t)}).
All this is used for the basic calculations of the paper,
the determination of the cohomological invariants for
the group schemes $\mu_p$
and $\Z/p$ (Propositions \ref{mup} and \ref{zp}).

Let $F$ be a field with a discrete valuation $v$. Let $O_F$
be the valuation ring $\{x\in F: v(x)\geq 0\}$, and let $k=O_F/\m$
be the residue field. Define the subgroup
of {\it unramified }classes in $H^i_{\et}(F,\Z/m(j))$ to be the image
of $H^i_{\et}(O_F,\Z/m(j))$. (More concretely, for $p=\ch(k)$,
in the description
of $H^{n+1}_{\et}(F,\Z/p(n))$ as a quotient of $\Omega^n_F$ (section
\ref{background}), the unramified subgroup
is the subgroup generated by elements
$a(db_1/b_1)\wedge\cdots\wedge(db_n/b_n)$ with
$a_i\in O_F$ and $b_1,\ldots,b_n\in O_F^*$.)
If $m$ is invertible in $k$,
then the subgroup of unramified classes is the kernel
of the {\it residue }homomorphism \cite[Part 1, section 7.9]{GMS}:
$$\partial_v\colon H^i_{\et}(F,\Z/m(j))\arrow H^{i-1}_{\et}(k,\Z/m(j-1)).$$

If $m$ is not invertible in $R$, what happens is more complicated,
but still manageable. If $F$ is complete with respect to the valuation
$v$, define the {\it tame }subgroup
in $H^i_{\et}(F,\Z/m(j))$ to be the kernel of the homomorphism
to $H^i_{\et}(F_{\tame},\Z/m(j))$, where $F_{\tame}$ is
the maximal tamely ramified extension of $F$. (An algebraic extension of
a complete discrete valuation field $F$
is {\it tame }if it is a union of finite
extensions such that
the extension of residue fields is separable
and the ramification degree is invertible in $k$.) For any discrete
valuation field $F$, not necessarily complete, define
the {\it tame }(or {\it tamely ramified})
subgroup of $H^i_{\et}(F,\Z/m(j))$ to be the inverse image
of the tame subgroup in $H^i_{\et}(F_v,\Z/m(j))$, writing $F_v$
for the completion.

The whole group $H^i$ is tamely ramified
if $m$ is invertible in $k$. For general $m$,
the residue homomorphism is not defined
on all of $H^i_{\et}(F,\Z/m(j))$, but only on the tamely
ramified subgroup \cite[Corollary 2.7]{Izhboldinrat}:
$$\partial_v\colon H^i_{\et,\tame}(F,\Z/m(j))\arrow H^{i-1}_{\et}(k,
\Z/m(j-1))$$
As a result, mod $p$ \etale motivic cohomology does not fit
into the framework of Rost's cycle modules \cite{Rost}.
On the good side, Theorem \ref{filtration} will say: (1) The unramified
subgroup of \etale motivic cohomology
is the kernel of the residue on the tamely ramified subgroup.
(2) There is a satisfactory description
of the quotient of \etale motivic cohomology by the tamely ramified
subgroup.

\begin{remark}
When $m=\text{char}(k)$,
Izhboldin calls our ``tamely ramified'' subgroup
of \etale motivic cohomology
the ``unramified'' subgroup \cite{Izhboldinrat}.
That has the confusing consequence that the residue
homomorphism is nontrivial on his ``unramified'' subgroup. Our
use of ``tamely ramified'' follows Kato \cite[Theorem 3]{Katocomplete}
and Auel-Bigazzi-B\"ohning-von Bothmer \cite[Remark 3.8]{Auel}.
It also agrees with the terminology
used for the Brauer group \cite[Proposition 6.63]{TW}. (Note that
the group $H^2_{\et}(k,\Z/m(1))$ is the subgroup of the Brauer group of $k$
killed by $m$, for any positive integer $m$.)
\end{remark}

When the discretely valued field $F$ is complete of characteristic $p>0$,
Izhboldin analyzed the ``wild quotient''
of $H^{n+1,n}(F)=H^{n+1}_{\et}(F,\Z/p(n))$; we generalize his
result (not assuming completeness)
as Theorem \ref{filtration}. To set this up, use the description
of $H^{n+1,n}(F)$ as a quotient of $\Omega^n_F$ from section
\ref{background}.
Define an increasing filtration of $H^{n+1,n}(F)$ by: for $i\geq 0$,
let $U_i$ be the subgroup of $H^{n+1,n}(F)$ generated by
elements of the form
$$f\frac{dg_1}{g_1}\wedge\cdots\wedge\frac{dg_n}{g_n}$$
with $f\in F$, $g_1,\ldots,g_n\in F^*$, and $v(f)\geq -i$.
Then $U_0$ is the tamely ramified subgroup of $H^{n+1,n}(F)$, and
it is clear that
$$0\subset U_0\subset U_1\subset \cdots,$$
with $\cup_{i\geq 0} U_i = H^{n+1,n}(F)$.

Let $t\in O_F$ be a uniformizer for $v$, and write $a\mapsto \o{a}$
for the surjection $O_F\arrow k$.
If $j>0$ and $j$ is prime to $p$, define a homomorphism
$$\Omega^n_k\arrow U_j/U_{j-1}$$
by
$$\o{a}\frac{d\o{b_1}}{\o{b_1}}\wedge\cdots\wedge\frac{d\o{b_n}}
{\o{b_n}} \mapsto \frac{a}{t^j} 
\frac{db_1}{b_1}\wedge\cdots\wedge\frac{db_n}
{b_n} \pmod{U_{j-1}},$$
for $a\in O_F$ and $b_1,\ldots,b_n\in O_F^*$.
Let $Z^n_k$ be the subgroup of closed forms in $\Omega^n_k$.
If $j>0$ and $p|j$, define a homomorphism
$$\Omega^n_k/Z^n_k \oplus \Omega^{n-1}_k/Z^{n-1}_k\arrow U_j/U_{j-1}$$
by (for the first summand)
$$\o{a}\frac{d\o{b_1}}{\o{b_1}}\wedge\cdots\wedge\frac{d\o{b_n}}
{\o{b_n}} \mapsto \frac{a}{t^j} 
\frac{db_1}{b_1}\wedge\cdots\wedge\frac{db_n}
{b_n} \pmod{U_{j-1}}$$
and (for the second summand)
$$\o{a}\frac{d\o{b_1}}{\o{b_1}}\wedge\cdots\wedge\frac{d\o{b_{n-1}}}
{\o{b_{n-1}}} \mapsto \frac{a}{t^j} \frac{dt}{t}\wedge
\frac{db_1}{b_1}\wedge\cdots\wedge\frac{db_{n-1}}
{b_{n-1}} \pmod{U_{j-1}},$$
where $a\in O_F$ and $b_1,\ldots,b_n\in O_F^*$.

It is straightforward to check that the homomorphisms above
are well-defined (although they depend on the choice of uniformizer $t$).
First check that that the element in $U_j/U_{j-1}$ associated
to given elements $\o{a}\in k$ and $\o{b_i}\in k^*$ is independent
of the choice of lifts to $O_F$. (For example, in the case $j>0$, $p\nmid j$,
it is clear that changing the lift of $\o{a}$ changes the result by
an element of $U_{j-1}$. Changing the lift of $\o{b_i}$ amounts to
multiplying $b_i$ by $1+e$ for some $e\in \m$; since
$d(1+e)/(1+e)= (e/(1+e))(de/e)$, where $e/(1+e)$ is in $\m$,
this change of lift changes the result by adding an element of $U_{j-1}$,
as we want.) To finish showing that the homomorphisms above
are well-defined,
use Kato's presentation of $\Omega^n_k$
\cite[section 1.3, Lemma 5]{Katogen}:

\begin{proposition}
\label{katopres}
For any field $k$ and natural number $n$,
the group of differentials $\Omega^n_k=\Omega^n_{k/\Z}$
is the quotient of
$k\otimes_{\Z}(k^*)^{\otimes n}$ by the relations:
$$[a,b_1,\ldots,b_n\}=0$$
if $a\in k$, $b_1,\ldots,b_n\in k^*$, and $b_i=b_j$ for some $i\neq j$;
and
$$[u+v,u+v,b_2,\ldots,b_n\}=[u,u,b_2,\ldots,b_n\}+[v,v,b_2,\ldots,b_n\}$$
if $u,v,u+v\in k^*$.
(The map from this quotient group
to $\Omega^n_k$ takes the symbol
$[a,b_1,\ldots,b_n\}$ to $a_1 (db_1/b_1)\wedge\cdots
\wedge (db_n/b_n)$.)
\end{proposition}

When $j>0$ and $p\nmid j$, it is straightforward from Proposition
\ref{katopres} to check that we have a well-defined homomorphism
$\Omega^n_k\arrow U_j/U_{j-1}$, above. When $j>0$ and $p|j$,
we can likewise see that we have a well-defined homomorphism
$\Omega^n_k/Z^n_k\oplus \Omega^{n-1}_k/Z^{n-1}_k\arrow U_j/U_{j-1}$,
using Cartier's theorem that,
for $k$ of characteristic $p>0$, the subgroup
$Z^n_k$ of closed forms in $\Omega^n_k$ is generated by the exact forms
together with the forms $a^p(db_1/b_1)\wedge\cdots\wedge (db_n/b_n)$
\cite[Lemma 1.5.1]{Izhboldinrat}. Our generalization 
of Izhboldin's result is:

\begin{theorem}
\label{filtration}
Let $F$ be a field of characteristic $p>0$ with a discrete
valuation $v$. Then $H^{n+1,n}(F)$ is the union of an increasing
sequence of subgroups $U_0\subset U_1\subset \cdots$,
with isomorphisms (depending on a choice of uniformizer in $F$):
$$U_j/U_{j-1}\cong \begin{cases} \Omega^n_k &\text{if }j>0\text{ and }
p\nmid j,\\
\Omega^n_k/Z^n_k \oplus \Omega^{n-1}_k/Z^{n-1}_k &\text{if }j>0
\text{ and }p|j.
\end{cases}$$
Moreover, there is a well-defined residue homomorphism
on $U_0=H^{n+1,n}_{\tame}(F)$, yielding an exact sequence
$$0\arrow H^{n+1,n}_{\nr}(F)\arrow H^{n+1,n}_{\tame}(F)
\xrightarrow[\partial_v]{}
H^{n,n-1}(k)\arrow 0,$$
where $H^{n+1,n}_{\nr}(F)$ is the unramified subgroup. Finally,
if the field $F$ is henselian (for example, complete)
with respect to $v$, then
$H^{n+1,n}_{\nr}(F)\cong H^{n+1,n}(k)$.
\end{theorem}

Without making a choice of uniformizer, the argument gives
the following canonical descriptions of $U_j/U_{j-1}$, which
we will not need:
$$U_j/U_{j-1}\cong \Omega^n_k\otimes_k (\m/\m^2)^{\otimes -j}$$
if $j>0$, $p\nmid j$, and
$$0\arrow (\Omega^n_k/Z^n_k)\otimes_k (\m/\m^2)^{\otimes -j}\arrow
U_j/U_{j-1}\arrow (\Omega^{n-1}_k/Z^{n-1}_k)\otimes_k (\m/\m^2)^{\otimes -j}
\arrow 0$$
if $j>0$, $p|j$.

\begin{proof}
When $F$ is complete, this was proved by Izhboldin
\cite[Theorem 2.5]{Izhboldinrat}. We address the henselian case
at the end. For any discretely valued field $F$,
write $F_v$ for the completion of $F$ with respect to $v$.
For brevity, write $U_j=U_j(F)$ and $N_j=U_j(F_v)$; thus
we know that $N_j/N_{j-1}$ is isomorphic to $\Omega^n_k$
for $j>0$, $p\nmid j$,
and to $\Omega^n_k/Z^n_k\oplus \Omega^{n-1}_k/Z^{n-1}_k$
for $j>0$, $p|j$.
There are obvious homomorphisms $U_j\arrow N_j$. It is clear
that $U_0$ is the tame subgroup of
$H^{n+1,n}(F)$; by definition, this statement reduces to the
corresponding fact for $F_v$.
We want to show that the homomorphism $U_j/U_{j-1}
\arrow N_j/N_{j-1}$ is an isomorphism for all $j>0$.

First, suppose that $j>0$ and $p\nmid j$. Fix a uniformizer $t$ for $F$.
From before
the theorem, we have homomorphisms
$$\Omega^n_k\arrow U_j/U_{j-1}\arrow N_j/N_{j-1}$$
whose composition is an isomorphism by Izhboldin.
To show that these homomorphisms are isomorphisms,
it suffices to show that our homomorphism $\Omega^n_k\arrow
U_j/U_{j-1}$ is surjective. 
Because $F^*=t^{\Z}\times O_F^*$, $U_j/U_{j-1}$ is generated by two types
of elements:
$(a/t^j)(db_1/b_1)\wedge\cdots\wedge(db_n/b_n)$ with
$a\in O_F$ and $b_1,\ldots,b_n\in O_F^*$, and elements
$(a/t^j)(dt/t)\wedge(db_2/b_2)\wedge\cdots\wedge(db_n/b_n)$
with $a\in O_F$ and $b_2,\ldots,b_n\in O_F^*$. The first
elements are clearly in the image of $\Omega^n_k$, by our
construction. For the second type of element, use that $p\nmid j$,
so that $d(-1/(jt^j))=(1/t^j)dt/t$. Therefore, for $a\in O_F$,
which we can assume is not zero,
and $b_2,\ldots,b_n\in O_F^*$,
$$d\bigg( -\frac{a}{jt^j}\frac{db_2}{b_2}\wedge\cdots\wedge
\frac{db_n}{b_n}\bigg) =
-\frac{a}{jt^j}\frac{da}{a}\wedge\frac{db_2}{b_2}\wedge\cdots\wedge
\frac{db_n}{b_n}
+\frac{a}{t^j}\frac{dt}{t}\wedge\frac{db_2}{b_2}\wedge\cdots\wedge
\frac{db_n}{b_n}.$$
Since exact forms represent zero in $H^{n+1,n}(F)$, it follows
that the element $(a/t^j)(dt/t)\wedge(db_2/b_2)\wedge\cdots\wedge(db_n/b_n)$
in $U_j/U_{j-1}$ that we are considering is equal to an element
$(a/(jt^j))(da/a)\wedge(db_2/b_2)\wedge\cdots\wedge(db_n/b_n)$.
If $a$ is in $O_F^*$, then this element is in the image of 
$\Omega^n_k$, as we want. On the other hand, if $a\in \m$,
then our element is in $U_{j-1}$, hence
zero in $U_j/U_{j-1}$. This completes
the proof that $U_j/U_{j-1}\cong \Omega^n_k$ for $j>0$, $p\nmid j$.

For $j>0$, $p|j$, we defined homomorphisms (before the theorem)
$$\Omega^n_k/Z^n_k \oplus \Omega^{n-1}_k/Z^{n-1}_k\arrow
U_j/U_{j-1} \arrow N_j/N_{j-1}$$
whose composition is an isomorphism. To show that these
homomorphisms are isomorphisms, it suffices to show
that $\Omega^n_k/Z^n_k\oplus \Omega^{n-1}_k/Z^{n-1}_k\arrow U_j/U_{j-1}$
is surjective. That is immediate from the definition of this
homomorphism. Indeed, since $F^*= t^{\Z}\times O_F^*$,
$U_j/U_{j-1}$ is generated
by two types of elements:
$(a/t^j)(db_1/b_1)\wedge\cdots\wedge(db_n/b_n)$ with
$a\in O_F$ and $b_1,\ldots,b_n\in O_F^*$, and
$(a/t^j)(dt/t)\wedge(db_2/b_2)\wedge\cdots\wedge(db_n/b_n)$
with $a\in O_F$ and $b_2,\ldots,b_n\in O_F^*$.

Thus we have determined the structure of $U_j/U_{j-1}$ for all
$j>0$. It is clear that $H^{n+1,n}(F)=\cup_{j\geq 0}U_j$.

Next, we show that the obvious homomorphism
$$H^{n+1,n}_{\tame}(F)/H^{n+1,n}_{\nr}(F)
\arrow H^{n,n-1}_{\tame}(F_v)/H^{n+1,n}_{\nr}(F_v)\cong H^{n,n-1}(k)$$
is an isomorphism. Here we define the unramified subgroup
$H^{n+1,n}_{\nr}(F)$ as the image of $H^{n+1,n}(O_F)$, or more
concretely as the subgroup generated by differential
forms $a(db_1/b_1)\wedge\cdots\wedge(db_n/b_n)$ with
$a_i\in O_F$ and $b_1,\ldots,b_n\in O_F^*$.

First, we define a homomorphism $H^{n,n-1}(k)\arrow
H^{n+1,n}_{\tame}(F)/H^{n+1,n}_{\nr}(F)$; it will be clear
that the composition
$$H^{n,n-1}(k)\arrow H^{n+1,n}_{\tame}(F)/H^{n+1,n}_{\nr}(F)
\arrow H^{n,n-1}(k)$$
is the identity. Namely, we map
$$\o{a}\frac{d\o{b_1}}{\o{b_1}}\wedge\cdots\wedge\frac{d\o{b_{n-1}}}
{\o{b_{n-1}}} \mapsto a \frac{dt}{t}\wedge
\frac{db_1}{b_1}\wedge\cdots\wedge\frac{db_{n-1}}
{b_{n-1}} \pmod{H^{n+1,n}_{\nr}(F)},$$
where $a\in O_F$ and $b_1,\ldots,b_{n-1}\in O_F^*$.

As in previous arguments, it is straightforward to check
that the resulting element of
$H^{n+1,n}_{\tame}(F)/H^{n+1,n}_{\nr}(F)$ does not depend
on the choice of lifts of $\o{a}\in k$ and $\o{b_1},\ldots,
\o{b_{n-1}}\in k^*$ to $O_F$. For brevity, we just write this out
for $\o{a}$. Namely, changing the lift
of $\o{a}$ changes the element of $H^{n+1,n}(F)$ by an expression of the form
$ct (dt/t)\wedge (db_1/b_1)\wedge\cdots (db_{n-1}/b_{n-1})$ with
$c\in O_F$ and $b_1,\ldots,b_{n-1}\in O_F^*$. We rewrite that
in $H^{n+1,n}(F)$ as:
\begin{align*}
ct\frac{dt}{t}\wedge 
\frac{db_1}{b_1}\wedge\cdots\wedge\frac{db_{n-1}}{b_{n-1}}
&= d\bigg[ ct \frac{db_1}{b_1}\wedge\cdots\wedge\frac{db_{n-1}}{b_{n-1}}\bigg]
-t\, dc\wedge \frac{db_1}{b_1}\wedge\cdots\wedge\frac{db_{n-1}}{b_{n-1}}\\
&= -ct \frac{dc}{c}\wedge
\frac{db_1}{b_1}\wedge\cdots\wedge\frac{db_{n-1}}{b_{n-1}},
\end{align*}
using that exact forms represent zero in $H^{n+1,n}(F)$. If $c$
is in $O_F^*$, then this element is unramified. Using that
$O_F$ is additively generated by $O_F^*$, we find that the element
above is always unramified, as we want. 

Thus we have a well defined function from $k\times (k^*)^{n-1}$
to the quotient group
$H^{n+1,n}_{\tame}(F)/H^{n+1,n}_{\nr}(F)$. It is clearly multilinear,
and so it gives a homomorphism from the abelian group
$k\otimes_{\Z} (k^*)^{\otimes n}$ to the latter quotient group. By
Proposition \ref{katopres}, the homomorphism factors
through $\Omega^{n-1}_k$ if the following elements map to zero:
$$[\o{a},\o{b_1},\ldots,\o{b_{n-1}}\}$$
with $\o{b_i}=\o{b_j}\in k^*$ for some $i\neq j$,
and 
$$[\o{u}+\o{v},\o{u}+\o{v},\o{b_2},\ldots,\o{b_{n-1}}\}
-[\o{u},\o{u},\o{b_2},\ldots,\o{b_{n-1}}\}-[\o{v},\o{v},\o{b_2},\ldots,
\o{b_{n-1}}\}$$
if $\o{u},\o{v},\o{u}+\o{v}\in k^*$. It is easy to check
that these elements map to zero, by choosing suitable
lifts (for example, take $b_i$ to be equal to $b_j$
when $\o{b_i}=\o{b_j}$ for some $i\neq j$).

Thus we have a well-defined homomorphism
from $\Omega^{n-1}_k$ to the quotient group
$H^{n+1,n}_{\tame}(F)/H^{n+1,n}_{\nr}(F)$.
To show that the homomorphism vanishes on exact $(n-1)$-forms,
it suffices to show that each element
of the form $[\o{a},\o{a},\o{b_2},\ldots,\o{b_{n-1}}\}$ maps to zero.
Those elements map to zero by definition of the homomorphism,
using that exact $n$-forms represent zero in $H^{n+1,n}(F)$.
Finally, to show that the homomorphism factors through
the quotient $H^{n,n-1}(k)$ of $\Omega^{n-1}_k$, it suffices
to show that $[\o{a}^p-\o{a},\o{b_1},
\ldots,\o{b_{n-1}}\}$ maps to zero. That holds because
forms $(a^p-a) (dt/t)\wedge(db_1/b_1)\wedge\cdots\wedge(db_{n-1}/b_{n-1})$
represent zero in $H^{n+1,n}(F)$.

Thus we have a well-defined homomorphism $\varphi$ from $H^{n,n-1}(k)$
to the quotient group
$H^{n+1,n}_{\tame}(F)/H^{n+1,n}_{\nr}(F)$. Composing this with
the residue homomorphism from the latter group to $H^{n,n-1}(k)$
(discussed earlier) gives the identity. Therefore, $\varphi$
is an isomorphism
if it is surjective. To prove surjectivity, use that
$H^{n+1,n}_{\tame}(F)$ is generated by elements
$a(db_1/b_1)\wedge\cdots\wedge(db_n/b_n)$ with $a\in O_F$
and $b_1,\ldots,b_n\in F^*$. Since $F^* = t^{\Z}\times O_F^*$,
$H^{n+1,n}_{\tame}(F)$ is in fact generated by elements
$a(db_1/b_1)\wedge\cdots\wedge(db_n/b_n)$ and
$a(dt/t)\wedge(db_2/b_2)\wedge\cdots\wedge(db_n/b_n)$ with $a\in O_F$
and $b_i\in O_F^*$. Elements of the first type are unramified, hence zero
in $H^{n+1,n}_{\tame}(F)/H^{n+1,n}_{\nr}(F)$, and elements of the second
type are in the image of $\varphi$. Thus $\varphi$ is an isomorphism.

Finally, when $F$ is henselian with respect to $v$, we want to show
that $H^{n+1,n}_{\nr}(F)\cong H^{n+1,n}(k)$. Here $H^{n+1,n}_{\nr}(F)$
is the subgroup of $H^{n+1,n}(F)$ generated by elements
of the form
$$a\frac{db_1}{b_1}\wedge\cdots\wedge\frac{db_n}{b_n}$$
with $a\in O_F$ and $b_1,\ldots,b_n\in O_F^*$.
We want to show that the map $H^{n+1,n}(k)\arrow H^{n+1,n}_{\nr}(F)$
given by the formula
$$\o{a}\frac{d\o{b_1}}{\o{b_1}}\wedge\cdots\wedge\frac{d\o{b_n}}
{\o{b_n}} \mapsto a
\frac{db_1}{b_1}\wedge\cdots\wedge\frac{db_n}{b_n},$$
for $a\in O_F$ and $b_1,\ldots,b_n\in O_F^*$,
is defined and an isomorphism.

We first show that given $\o{a}\in k$
and $\o{b_1},\ldots,\o{b_n}\in k^*$, the choice of lifts
to $O_F$ does not affect the right side in $H^{n+1,n}_{\nr}(F)$.
The choice of lift $a$ does not matter, because
every element of $\m\subset O_F$ can be written as $u^p-u$
for some $u\in F$, using that $O_F$ is henselian
\cite[Theorem I.4.2($d'$)]{Milne}.
Next, the choice of lift $b_1$ (say) does not matter,
because for $c_1\neq 0\in \m$ and $b_1=1+c_1$, with elements $a\in O_F$
and $b_2,\ldots,b_n\in O_F^*$,
$$a \frac{d(1+c_1)}{1+c_1}\wedge\frac{db_2}{b_2}\wedge
\cdots\wedge\frac{db_n}{b_n}
=\frac{ac_1}{1+c_1}\; \frac{dc_1}{c_1}\wedge\frac{db_2}{b_2}
\wedge\cdots\wedge\frac{db_n}{b_n},$$
which is zero in $H^{n+1,n}_{\nr}(F)$ because
$ac_1/(1+c_1)$ is in $\m$ and hence can be written
as $u^p-u$ for some $u\in F$.

Using Proposition \ref{katopres}, it follows that the formula above
gives a well-defined map $\Omega^n_k\arrow H^{n+1,n}_{\nr}(F)$.
Finally, using the description of $H^{n+1,n}(k)$ as the cokernel
of $\mathcal{P}\colon \Omega^n_k\arrow \Omega^n_k/d\Omega^{n-1}_k$,
it follows that the formula above gives a well-defined map
$H^{n+1,n}(k)\arrow H^{n+1,n}_{\nr}(F)$. This is surjective
by definition. Injectivity follows from Izhboldin's result
that the composed map to $H^{n+1,n}_{\nr}$ of the completion $F_v$
is an isomorphism \cite[Corollary 2.7]{Izhboldinrat}.
\end{proof}

Finally, we state Izhboldin's calculation of the mod $p$
\etale motivic cohomology of the rational function field
in one variable over any field of characteristic $p$
\cite[Theorem 4.5]{Izhboldinrat}. For example, this result 
gives the $p$-torsion in the Brauer group of $k(t)$, generalizing
the Faddeev exact sequence (which addresses the special
case where $k$ is perfect) \cite[Corollary 6.4.6]{GS}.
Our terminology is slightly different
from Izhboldin's, but the translation is straightforward.

\begin{theorem}
\label{k(t)}
Let $k$ be a field of characteristic $p>0$, and let $n$ be a natural
number. Let $S$ be the set of closed points in $\P^1_k$. For $v\in S$,
write $k(t)_v$ for the completion of the field $k(\P^1)=k(t)$ at $v$. Then:

(1) The natural homomorphism
$$H^{n+1,n}(k(t))\arrow \oplus_{v\in S}
H^{n+1,n}(k(t)_v)/H^{n+1,n}_{\tame}(k(t)_v)$$
is surjective. The wild quotients on the right are described
by Theorem \ref{filtration}.

(2) The kernel of that surjection,
which we call $H^{n+1,n}_{\tame}(k(t))$, fits into
an exact sequence:
$$0\arrow H^{n+1,n}(k)\arrow H^{n+1,n}_{\tame}(k(t))
\arrow \oplus_{v\in S}H^{n,n-1}(k(v))\arrow H^{n,n-1}(k)\arrow 0.$$
Here $k(v)$ denotes the residue field of $\P^1_k$ at a closed point $v$,
and the homomorphism to $H^{n,n-1}(k(v))$ is the residue defined above.
\end{theorem}

\section{Finite groups}

In this section, we show that the mod $p$ cohomological invariants
for a finite group, viewed as a group scheme
over a field $k$ of characteristic $p$,
are nearly trivial when $k$ is perfect. By contrast,
more general finite group schemes can have richer
mod $p$ cohomological invariants. It would be interesting
to find out how far the results of this section extend
to imperfect fields; see section \ref{symmsect} for the case
of the symmetric groups.

\begin{theorem}
\label{smoothmilnor}
Let $G$ be a smooth affine group over a
field $k$ of characteristic $p>0$. For any $n\geq 0$,
all invariants of $G$ over $k$ with values in $H^{n,n}$
are constant. That is, $\Inv^{n,n}_k(G)=H^{n,n}(k)$.
\end{theorem}

\begin{proof}
Let $\alpha$ be a normalized invariant for $G$ of degree $(n,n)$.
Let $E$ be any $G$-torsor over a field $F/k$; we want to show
that $\alpha(E)=0$. Since $G_F$ is smooth over $F$,
$E$ becomes trivial over the separable closure $F_s$. So the image
of $\alpha(E)$ in $H^{n,n}(F_s)$ is zero. But $H^{n,n}(F)\arrow
H^{n,n}(F_s)$ is injective, by Bloch and Kato's isomorphism
$H^{n,n}(F)\cong \Omega^n_{\log, F}\subset \Omega^n_F$
(discussed in section \ref{background}). So
$\alpha(E)=0$, as we want.
\end{proof}

In particular, finite groups have no normalized
mod $p$ cohomological invariants
of bidegree $(n,n)$. We now check that this is also true
(over a perfect base field)
in the other possible bidegrees, $(n+1,n)$, except for bidegree $(1,0)$
(which is described in Theorem \ref{degree1}).

\begin{theorem}
\label{h1fingroup}
Let $G$ be a finite group, and let $k$ be a perfect
field of characteristic $p>0$. Then $\Inv^{n+1,n}_{k}(G)=0$
for all $n\geq 1$.
\end{theorem}

\begin{proof}
Let $V$ be a faithful representation of $G$ over $\F_p$ (for example,
the regular representation). Then
the open subset $U$ of $V$ on which $G$ acts freely is nonempty,
and there is a quotient variety $U/G$ over $\F_p$. Consider the
Frobenius morphism $F\colon U/G\arrow U/G$, which is a morphism
over $\F_p$. The pullback by $F$ of the $G$-torsor $U\arrow U/G$
is isomorphic to the same $G$-torsor, since the commutative
diagram
$$\xymatrix@C-10pt@R-10pt{
U \ar[r]_F\ar[d] & U\ar[d] \\
U/G \ar[r]_F & U/G
}$$
is a pullback of $G$-torsors. (This uses that the elements of $G=G(\F_p)$
give $\F_p$-automorphisms of $U$, which therefore commute
with Frobenius.)

Let $F_1\colon (U/G)_k\arrow (U/G)_k$ be the base change to $k$
of the morphism $F$ over $\F_p$. It follows from the previous paragraph
that the pullback by $F_1$ of the $G$-torsor $U_k\arrow (U/G)_k$
is isomorphic to the same $G$-torsor. Since $k$ is perfect,
we can also identify $F_1$ (after an automorphism of the scheme
$(U/G)_k$) with the absolute Frobenius morphism on $(U/G)_k$. As a result,
the pullback $F_1^*\colon k(U/G)\arrow k(U/G)$ sends
$k(U/G)^*$ into $(k(U/G)^*)^p$. Therefore, $F_1^*$ acts by
zero on $H^{n+1,n}(F)$ for all $n\geq 1$, by the interpretation
in terms of differential forms (section \ref{background}):
the pullback of a form
$db/b$ is of the form $d(c^p)/c^p=0$.

As a result, every invariant for $G$-torsors in $H^{n+1,n}$ with
$n\geq 1$ is zero in $H^{n+1,n}(k(U/G))$. Since the $G$-torsor
over $U/G$ is versal (section \ref{background}),
it follows that every such invariant is zero.
\end{proof}

\section{Invariants of $\mu_p$}

In this section, we use Theorem \ref{filtration}
to compute the cohomological invariants
of the group scheme $\mu_p$ of $p$th roots of unity
over any field of characteristic $p$. More generally,
we find the invariants for the product of $\mu_p$ with any group scheme.

The invariants for $(\mu_p)^r$ with values in $H^{n,n}$ were computed
by Vial, as part of his determination of the operations
on Milnor $K$-theory of fields \cite[Theorem 3.4]{Vial}:

\begin{theorem}
\label{vialops}
Let $k$ be a field of characteristic $p>0$, and let $n$ and $r$
be natural numbers. Then
$$\Inv^{n,n}_k((\mu_p)^r)\cong
\oplus_{I\subset \{1,...,r\}} H^{n-|I|,n-|I|}(k).$$
In more detail, every invariant for $(\mu_p)^r$ over $k$
with values in $H^{n,n}$ has the form
$$u(\alpha_1,\ldots,\alpha_r)=\sum_{I\subset \{1,\ldots,r\}}
c_I\prod_{i\in I}\alpha_i$$
for some elements $c_I\in H^{n-|I|,n-|I|}(k)$. Here
$\alpha_1,\ldots,\alpha_r$ are $\mu_p$-torsors
over a field extension $E/k$, and
we use the identification $H^1(E,\mu_p)\cong H^{1,1}(E)$.
\end{theorem}

We now find the invariants of $\mu_p$ with values in $H^{n+1,n}$.

\begin{proposition}
\label{mup}
Let $k$ be a field of characteristic $p>0$. Then
$$\Inv^{n+1,n}_k(\mu_p)\cong H^{n+1,n}(k)\oplus H^{n,n-1}(k).$$
In more detail, every invariant
for $\mu_p$ over $k$ with values in $H^{n+1,n}$ has the form
$$u(\beta)=v+w\beta$$
for some $v\in H^{n+1,n}(k)$ and $w\in H^{n,n-1}(k)$. Here $\beta$
denotes any $\mu_p$-torsor over a field extension $E/k$, and
we use the identification $H^1(E,\mu_p)\cong H^{1,1}(k)$.
\end{proposition}

\begin{proof}
Let $u$ be an invariant for $\mu_p$ over $k$ with values
in $H^{n+1,n}$.
Let $\{t\}$ denote the $\mu_p$-torsor over the field $k(t)$ associated
to $t\in k(t)^*/(k(t)^*)^p\cong H^1(k(t),\mu_p)$. Then $u$ gives
an element $u(\{t\})\in H^{n+1,n}(k(t))$. Here $\{t\}$ is a versal
torsor for $\mu_p$, ccrresponding to the $\mu_p$-torsor
$U\arrow U/\mu_p$ where $U=A^1_k-0$; so $u$ is determined
by the element $u(\{t\})$ in $H^{n+1,n}(k(t))$.

We know that $u(\{t\})$ is unramified on $U/\mu_p\cong A^1_k-0$ by Theorem
\ref{bm}. Let us show that it is also tamely ramified at
$t=0$ in $\P^1_k$; the same argument gives that $u(\{t\})$
is tamely ramified at $t=\infty$.
If $u(\{t\})$ is not tamely ramified
at $t=0$, then $u(\{t\})$ is in
$U_j-U_{j-1}$ for some $j>0$, with respect to the valuation
$t=0$ on $k(t)$, in the notation of section
\ref{ramsection}. Suppose first that $p\nmid j$;
then, by Theorem \ref{filtration}, we can write
$$u(\{t\})=\sum \frac{a}{t^j}\frac{db_1}{b_1}\wedge\cdots\wedge
\frac{db_n}{b_n}\pmod{U_{j-1}}$$
with $a$ in the local ring $O_{A^1,0}$
and $b_1,\ldots,b_n\in O_{A^1,0}^*$. (The expression
is meant to indicate a finite sum with 
$a,b_1,\ldots,b_n$ denoting different functions in each term.)

We know that $u(\{t\})$ is balanced, meaning that its pullback
by the two morphisms $(U\times U)/\mu_p\arrow U/\mu_p$
are equal (Theorem \ref{bm}). We can identify
the function field of $(U\times U)/\mu_p$ with the rational
function field $k(x,y)$, and balancedness
means that $u(\{x^py\})=u(\{y\})$. (This is clear directly, since
$x^py$ and $y$ define isomorphic $\mu_p$-torsors over $k(x,y)$.)
So we must have
$$\sum \frac{a(x^py)}{x^{pj}y^j}\frac{db_1(x^py)}{b_1(x^py)}\wedge\cdots\wedge
\frac{db_n(x^py)}{b_n(x^py)}=
\sum \frac{a(y)}{y^j}\frac{db_1(y)}{b_1(y)}\wedge\cdots\wedge
\frac{db_n(y)}{b_n(y)}$$
in $H^{n+1,n}(k(x,y))$.
The element on the right is clearly unramified along the divisor
$x=0$ in $A^2_k=\Spec\, k[x,y]$, and so the element on the left
is also unramified along $x=0$. That element is visibly
in $U_{pj}$ with respect to the valuation $x=0$, and so its
class in $U_{pj}/U_{pj-1}$ must be zero. Since the residue
field for that valuation on $k(x,y)$ is $k(y)$,
Theorem \ref{filtration} gives that the form
$$\sum \frac{a(0)}{y^j}\frac{db_1(0)}{b_1(0)}\wedge\cdots\wedge
\frac{db_n(0)}{b_n(0)}$$
in $\Omega^n_{k(y)}$ is {\it closed}. That is,
$$0=\sum \frac{1}{y^j}da(0)\wedge \frac{db_1(0)}{b_1(0)}\wedge\cdots\wedge
\frac{db_n(0)}{b_n(0)}
-j\sum \frac{a(0)}{y^{j+1}}dy\wedge\frac{db_1(0)}{b_1(0)}\wedge\cdots\wedge
\frac{db_n(0)}{b_n(0)}$$
in $\Omega^{n+1}_{k(y)}$. The differential forms on $k(y)$ are easy
to describe:
$$\Omega^{n+1}_{k(y)}\cong k(y)\otimes_k \Omega^{n+1}_k
\oplus dy\cdot k(y)\otimes_k \Omega^n_k.$$
So both sums in the expression above must be zero.
Since we are assuming that $p\nmid j$, it follows that both
$\sum da(0)\wedge (db_1(0)/b_1(0))\wedge \cdots$ in $\Omega^{n+1}_k$
and $\sum a(0) (db_1(0)/b_1(0))\wedge\cdots$ in $\Omega^n_k$
are zero. The second statement means that the element
$u((t))\in U_j=U_j(k(t))$ is actually in $U_{j-1}$, contradicting our 
assumption.

Now suppose that $u(\{t\})$ is in $U_j-U_{j-1}$ (with respect
to the valuation $t=0$ on $k(t)$) with
$j>0$ and $p|j$. Because $k(t)^*=t^{\Z}\times O_{A^1,0}^*$, we can
write $u(\{t\})$ as a sum of two types of terms:
$$u(\{t\})=
\sum \frac{a}{t^j}\frac{db_1}{b_1}\wedge\cdots\wedge
\frac{db_n}{b_n}+
\sum \frac{e}{t^j}\frac{dt}{t}\frac{dc_1}{c_1}\wedge\cdots\wedge
\frac{dc_n}{c_{n-1}}$$
with $a(t)$ and $e(t)$ in $O_{A^1,0}$ and $b_i(t)$ and $c_i(t)$
in $O_{A^1,0}^*$.

As in the previous argument, the elements $x^py$ and $y$
in $k(x,y)^*$ determine isomorphic $\mu_p$-torsors
over $k(x,y)$, and so the pullbacks of $u(\{t\})$ to
$H^{n+1,n}(k(x,y))$ by $t=y$ and $t=x^py$ must be equal.
The first pullback is clearly unramified along the divisor
$x=0$ in $A^2_k=\Spec\, k[x,y]$, and so the second pullback must
also be. That is,
$$\sum \frac{a(x^py)}{x^{pj}y^j}\frac{db_1(x^py)}{b_1}(x^py)\wedge\cdots\wedge
\frac{db_n(x^py)}{b_n(x^py)}+
\sum \frac{e(x^py)}{x^{pj}y^j}\frac{dy}{y}\wedge\frac{dc_1(x^py)}{c_1(x^py)}
\wedge\cdots\wedge
\frac{dc_{n-1}(x^py)}{c_{n-1}(x^py)}$$
in $H^{n+1,n}(k(x,y))$ is unramified along $x=0$. It is visibly
in $U_{pj}$ with respect to the valuation $x=0$, and so its class
in $U_{pj}/U_{pj-1}$ must be zero. By Theorem \ref{filtration},
this means that the form
$$\sum \frac{a(0)}{y^j}\frac{db_1(0)}{b_1}(0)\wedge\cdots\wedge
\frac{db_n(0)}{b_n(0)}+
\sum \frac{e(0)}{y^j}\frac{dy}{y}\wedge\frac{dc_1(0)}{c_1(0)}
\wedge\cdots\wedge
\frac{dc_{n-1}(0)}{c_{n-1}(0)}$$
in $\Omega^n_{k(y)}$ is closed. That is, using that $p|j$,
$$0=\sum \frac{1}{y^j}da(0)\wedge\frac{db_1(0)}{b_1}(0)\wedge\cdots\wedge
\frac{db_n(0)}{b_n(0)}+
\sum \frac{1}{y^j}de(0)\wedge\frac{dy}{y}\wedge\frac{dc_1(0)}{c_1(0)}
\wedge\cdots\wedge
\frac{dc_{n-1}(0)}{c_{n-1}(0)}$$
in $\Omega^{n+1}_{k(y)}$. Since
$$\Omega^{n+1}_{k(y)}\cong k(y)\otimes_k \Omega^{n+1}_k
\oplus dy\cdot k(y)\otimes_k \Omega^n_k,$$
it follows that the form $\sum da(0)\wedge (db_1(0)/b_1(0))\wedge\cdots$
is zero in $\Omega^{n+1}_k$ and $\sum de(0)\wedge (dc_1(0)/c_1(0))
\wedge\cdots$ is zero in $\Omega^n_k$. That is,
$\sum a(0)\wedge (db_1(0)/b_1(0))\wedge\cdots$
in $\Omega^{n}_k$ is closed, and $\sum e(0)\wedge (dc_1(0)/c_1(0))
\wedge \cdots$ in $\Omega^{n-1}_k$ is closed. Since $p|j$, this says
exactly (by Theorem \ref{filtration}) that the element
$$u(\{t\})=
\sum \frac{a}{t^j}\frac{db_1}{b_1}\wedge\cdots\wedge
\frac{db_n}{b_n}+
\sum \frac{e}{t^j}\frac{dt}{t}\frac{dc_1}{c_1}\wedge\cdots\wedge
\frac{dc_{n-1}}{c_{n-1}}$$
in $H^{n+1,n}(k(t))$ is zero in $U_j/U_{j-1}$, contradicting
our assumption.

Thus we have shown that $u(\{t\})$ in $H^{n+1}(k(t))$ is tamely
ramified at $t=0$ in $\P^1_k$. By the same argument, it is tamely
ramified at $t=\infty$. By Theorem \ref{k(t)},
the subgroup of elements of $H^{n+1,n}(k(t))$ that are 
unramified on $A^1_k-0$ and tamely ramified at 0 and $\infty$
is isomorphic to $H^{n+1,n}(k)\oplus H^{n,n-1}(k)$. Thus
$\Inv^{n+1,n}_k(\mu_p)$ injects into that direct sum. Since
we already know invariants for $\mu_p$ that give all elements
of that direct sum, we have
$$\Inv^{n+1,n}_k(\mu_p)=H^{n+1,n}(k)\oplus H^{n,n-1}(k).$$
\end{proof}

\begin{proposition}
\label{mupprod}
Let $H$ be an affine group scheme of finite type over a field $k$.
Then
$$\Inv^{n+1,n}_k(\mu_p\times H)\cong \Inv_k^{n+1,n}(H)\oplus
\Inv_k^{n,n-1}(H).$$
The isomorphism sends invariants $a$ and $b$ for $H$ (with $a$ of 
bidegree $(n+1,n)$ and $b$ of bidegree $(n,n-1)$) to 
$a+b\{t\}$, where $\{t\}$ is the obvious invariant
for $\mu_p$ of bidegree $(1,1)$, coming from the isomorphism
$H^1(F,\mu_p)\cong H^{1,1}(F)$ for fields $F$ over $k$.
\end{proposition}

\begin{proof}
Let $V$ be a $k$-vector space on which $H$ acts by affine
transformations, and suppose that $H$ acts freely
on a nonempty open subset $U$ of $V$ and the quotient
scheme $U/H$ exists. (Such pairs $(V,U)$ do exist
\cite[Remark 2.7]{Totarobook}.)

Let $u$ be an invariant of $\mu_p\times H$ over $k$ with values
in $H^{n+1,n}$. For any field $L$ over $k$ and any $H$-torsor
$\beta$ over $L$, we get an invariant $u_{\beta}$ of $\mu_p$
over $L$ with values in $H^{n+1,n}$ by defining
$$u_{\beta}(\alpha)=u(\alpha,\beta)$$
for any $\mu_p$-torsor $\alpha$ over an extension field
of $L$. By Proposition \ref{mup}, there are unique elements
$v\in H^{n+1,n}(L)$ and $w\in H^{n,n-1}(L)$ such that
$u_{\beta}(\alpha)=v+w\alpha$ for all $\mu_p$-torsors
$\alpha$ over fields over $L$. Here we are identifying
$H^1(E,\mu_p)$ with $H^{1,1}(E)$, for fields $E$ over $L$.

By that uniqueness, $v$ and $w$ are invariants of $H$-torsors
$\beta$ on fields over $k$. These invariants satisfy (and are
characterized uniquely by): for every $(\mu_p\times H)$-torsor
$(\alpha,\beta)$ on a field $E$ over $k$,
$$u(\alpha,\beta)=v(\beta)+w(\beta)\alpha.$$
Thus every invariant for $\mu_p\times H$ has this form,
with the invariants $v$ and $w$ uniquely determined.
Conversely, for any invariants $v$ and $w$ for $H$ over $k$,
the formula above defines an invariant for $\mu_p\times H$.
Thus we have shown that 
$$\Inv^{n+1,n}_k(\mu_p\times H)\cong \Inv_k^{n+1,n}(H)\oplus
\Inv_k^{n,n-1}(H).$$
\end{proof}

\section{Invariants of $\Z/p$}

Next, we find the cohomological invariants of $\Z/p$. When $k$
is perfect, this was mostly done
in Theorem \ref{h1fingroup}.
Here we consider any field of characteristic $p$, as is needed
for inductive arguments. More generally,
we find the invariants for the product of $\Z/p$ with any group.
Combining this with Proposition \ref{mupprod},
we determine all invariants of the group scheme
$(\Z/p)^r\times (\mu_p)^s$ (Theorem \ref{abgp}).

\begin{proposition}
\label{zp}
Let $k$ be a field of characteristic $p>0$. Then
$$\Inv^{n+1,n}_k(\Z/p)\cong H^{n+1,n}(k)\oplus H^{n,n}(k).$$
Explicitly, every invariant for $\Z/p$ over $k$ with values
in $H^{n+1,n}$ has the form $u(x)=v+xw$ for some $v$
in $H^{n+1,n}(k)$ and $w$ in $H^{n,n}(k)$. Here $x$
denotes the class of any $\Z/p$-torsor over a field $E/k$
in $H^1(E,\Z/p)\cong H^{1,0}(E)$.
\end{proposition}

\begin{proof}
Let $G=\Z/p$ act freely on the affine line $U$ over $k$ by translations.
Then $U\arrow U/G\cong A^1$ is a versal torsor $\xi$ for $G$.
Let $u$ be any cohomological invariant for $G$ over $k$ with values
in $H^{n+1,n}$; then $u$
is determined by $u(\xi)$ in $H^{n+1,n}(k(U/G))=H^{n+1,n}(k(t))$.

A cohomological invariant for any group scheme over $k$
is called {\it normalized }if
it is equal to zero on the trivial torsor over $k$.
The group of invariants
with values in $H^{n+1,n}$
is (clearly) the direct sum of the constant invariants $H^{n+1,n}(k)$
and the normalized invariants.

So it suffices to consider a normalized invariant $u$ for $G=\Z/p$
over $k$. Since the $G$-torsor $\xi$ over $U/G$ pulls back
to a trivial torsor over $U$, $u(\xi)$ in $H^{n+1,n}(k(U/G))$
pulls back to zero in $H^{n+1,n}(k(U))$. We now use the following
result of Izhboldin's \cite[Theorem B]{Izhboldintorsion}.

\begin{theorem}
\label{cyclic}
Let $F$ be a field of characteristic $p>0$, and let $E/F$ be a cyclic
extension of degree $p$. Then the sequence
$$H^{n,n}(F) \arrow H^{n+1,n}(F) \arrow H^{n+1,n}(E)$$
is exact. Here the second homomorphism is the obvious pullback,
and the first homomorphism is the product with the class of $E/F$
in $H^{1,0}(F)$.
\end{theorem}

It follows that $u(\xi)=[t]v$ for some $v$ in $H^{n,n}(k(t))\cong
\Omega^n_{\log,k(t)}$.
(Here we use that the $\Z/p$-covering $U\arrow U/G\cong A^1_k$
corresponds
to the element $t\in k(t)/\mathcal{P}(k(t))\cong H^{1,0}(k(t))$.)
In the description of $H^{n+1,n}(k(t))$ by differential forms
(section \ref{background}), it follows
that $u(\xi)$ is a sum $\sum t(da_1/a_1)\wedge\cdots\wedge(da_n/a_n)$
with $a_1,\ldots,a_n\in k(t)^*$. In coordinates $y=1/t$, this says
that $u(\xi)= \sum (1/y)(da_1/a_1)\wedge\cdots\wedge(da_n/a_n)$
with $a_i\in k(y)^*$.
Because $1/y$ has only a simple pole at $y=0$, the element
$u(\xi)$ is not too ramified at the point $y=0$ (corresponding to $t=\infty$)
in $\P^1_k$. Namely,
in the notation of section \ref{ramsection},
$u(\xi)$ is in $U_1$ with respect to the valuation $y=0$
on $k(y)=k(t)$.

Using that $k(y)^*=y^{\Z}\times O_{A^1_y,0}^*$, we can rewrite $u(\xi)$
as
$$u(\xi)=\sum \frac{1}{y}\frac{db_1}{b_1}\wedge\cdots\wedge
\frac{db_n}{b_n}
+\sum \frac{1}{y}\frac{dy}{y}\wedge\frac{dc_1}{c_1}
\wedge\cdots\wedge \frac{dc_{n-1}}{c_{n}},$$
with $b_i,c_i$ units at $y=0$. The forms in the second sum here are exact,
being equal to $d(-(1/y)(dc_1/c_1)\wedge\cdots\wedge(dc_{n-1}/c_{n-1}))$.
So $u(\xi)$ in $H^{n+1,n}(k(y))$ is represented by the form
$\sum (1/y)(db_1/b_1)\wedge\cdots\wedge(db_{n}/b_{n})$
with $b_i\in O_{A^1_y,0}^*$. By the formula for the isomorphism
$U_1/U_0\cong \Omega^1_k$ (Theorem \ref{filtration}) associated
to the choice of uniformizer $y$, it follows
that the class of $u(\xi)$ in $U_1/U_0\cong \Omega^n_k$
is in $\Omega^n_{\log,k}\cong H^{n,n}(k).$

We know that each element $\sigma$ of $H^{n,n}(k)$ gives a normalized
cohomological invariant for $G=\Z/p$ over $k$ with values
in $H^{n+1,n}$, by the product $H^1_{\et}(k,G)\times H^{n,n}(k)\arrow
H^{n+1,n}(k)$. It is immediate that $\sigma(\xi)$
in $H^{n+1,n}(k(t))$ has class (with respect to the valuation
$t=\infty$) in $U_1/U_0$ equal to $\sigma$. So, by subtracting
off an invariant of this form, we can assume that our normalized
invariant $u$ has the property that $u(\xi)$ 
in $H^{n+1,n}(k(t))$ has class in $U_1/U_0$
(at $t=\infty$) equal to zero. Equivalently, $u(\xi)$ is tamely
ramified at $t=\infty$. We want to show that a normalized invariant with
this property is zero.

We know that $u(\xi)$ in $H^{n+1,n}(k(t))$
is unramified over $U/G=A^1_k=\Spec \, k[t]$,
by Theorem \ref{bm}. By Theorem \ref{k(t)},
since $u(\xi)$ is unramified on $A^1_k$ and tamely ramified at
$t=\infty$, it is in fact unramified on all of $\P^1_k$ and comes
from an element of $H^{n+1,n}(k)$. But we took $u$ to be a normalized
invariant, and so $u(\xi)$ pulls back to zero in $H^{n+1,n}(k(U))$,
whereas pullback to $H^{n+1,n}(k(U))$ has trivial kernel on the subgroup
$H^{n+1,n}(k)\subset H^{n+1,n}(k(U/G))$. So $u(\xi)=0$ and hence
$u=0$. Thus the only invariants for $G=\Z/p$ are those listed.
\end{proof}

\begin{proposition}
\label{zpprod}
Let $H$ be an affine group scheme of finite type over a field
$k$ of characteristic $p>0$. Then
$$\Inv^{n+1,n}(\Z/p\times H)\cong \Inv^{n+1,n}_k(H)\oplus \Inv^{n,n}_k(H).$$
Explicitly, every invariant for $\Z/p\times H$ over $k$ with values
in $H^{n+1,n}$ has the form $u(\alpha,\beta)=v(\beta)+\alpha \, w(\beta)$
for some invariants $v$ of $H$ in $H^{n+1,n}$ and $w$ of $H$
in $H^{n,n}$. Here $\alpha$
is the class of any $\Z/p$-torsor over a field $E/k$
in $H^1(E,\Z/p)\cong H^{1,0}(E)$.
\end{proposition}

\begin{proof}
Let $V$ be a $k$-vector space on which $H$ acts by affine
transformations, and suppose that $H$ acts freely
on a nonempty open subset $U$ of $V$ and the quotient
scheme $U/H$ exists. (Such pairs $(V,U)$ do exist
\cite[Remark 2.7]{Totarobook}.)

Let $u$ be an invariant of $\Z/p\times H$ over $k$ with values
in $H^{n+1,n}$. For any field $L$ over $k$ and any $H$-torsor
$\beta$ over $L$, we get an invariant $u_{\beta}$ of $\Z/p$
over $L$ with values in $H^{n+1,n}$ by defining
$$u_{\beta}(\alpha)=u(\alpha,\beta)$$
for any $\Z/p$-torsor $\alpha$ over an extension field
of $L$. By Proposition \ref{mup}, there are unique elements
$v\in H^{n+1,n}(L)$ and $w\in H^{n,n}(L)$ such that
$u_{\beta}(\alpha)=v+w\alpha$ for all $\Z/p$-torsors
$\alpha$ over fields over $L$. Here we are identifying
$H^1(E,\Z/p)$ with $H^{1,0}(E)$, for fields $E$ over $L$.

By that uniqueness, $v$ and $w$ are invariants of $H$-torsors
$\beta$ on fields over $k$. These invariants satisfy (and are
characterized uniquely by): for every $(\Z/p\times H)$-torsor
$(\alpha,\beta)$ on a field $E$ over $k$,
$$u(\alpha,\beta)=v(\beta)+w(\beta)\alpha.$$
Thus every invariant for $\Z/p\times H$ has this form,
with the invariants $v$ and $w$ uniquely determined.
Conversely, for any invariants $v$ and $w$ for $H$ over $k$,
the formula above defines an invariant for $\Z/p\times H$.
Thus we have shown that 
$$\Inv^{n+1,n}_k(\Z/p\times H)\cong \Inv_k^{n+1,n}(H)\oplus
\Inv_k^{n,n}(H).$$
\end{proof}

Combining several earlier results, we now compute
all cohomological invariants of the group scheme $(\Z/p)^r\times (\mu_p)^s$.

\begin{theorem}
\label{abgp}
Let $k$ be a field of characteristic $p>0$, and let $r,s,n$ be natural
numbers. Then every cohomological invariant for $(\Z/p)^r
\times (\mu_p)^s$ over $k$ with values in $H^{n+1,n}$ is
of the form
$$u([a_1],\ldots,[a_r],\{b_1\},\ldots,\{b_s\})=
\sum_{I\subset \{1,\ldots,s\}}c_I\prod_{i\in I}\{b_i\}
+\sum_{j=1}^r [a_j]\sum_{I\subset \{1,\ldots,s\}}e_{j,I}\prod_{i\in I}\{b_i\}
$$
for some (unique) elements
$c_I$ in $H^{n-|I|+1,n-|I|}(k)$ and $e_{j,I}$ in $H^{n-|I|,n-|I|}(k)$.
That is,
$$\Inv_k((\Z/p)^r\times (\mu_p)^s)\cong
\oplus_{I\subset \{1,\ldots,s\}}H^{n-|I|+1,n-|I|}(k)
\oplus \oplus_{j=1}^r \oplus_{I\subset \{1,\ldots,s\}}H^{n-|I|,n-|I|}(k).$$
\end{theorem}

\begin{proof}
The group scheme $(\Z/p)^r$ is smooth over $k$,
and so all its invariants in $H^{n,n}$ are constant (Theorem
\ref{smoothmilnor}). Applying Proposition \ref{zpprod} (on products
with $\Z/p$),
we find that
$$\Inv^{n+1,n}_k((\Z/p)^r)\cong 
H^{n+1,n}(k) \oplus \oplus_{j=1}^r H^{n,n}(k).$$
Applying Proposition \ref{mupprod} (on products with $\mu_p$)
gives the invariants for $(\Z/p)^r\times (\mu_p)^s$.
\end{proof}

\section{Symmetric groups}
\label{symmsect}

For all finite groups (as opposed to more general finite group schemes),
it may be possible to determine the mod $p$
cohomological invariants over all
fields of characteristic $p$, not just perfect fields
as in Theorem \ref{h1fingroup}. Perhaps all invariants
come from the abelianization of the group. In this section, we prove
this in the case of the symmetric groups.

Equivalently, we determine the cohomological invariants
of \etale algebras in characteristic 2. There are analogies
with Serre's calculation in characteristic not 2. Regardless
of the characteristic, all invariants of \etale algebras
with odd-primary coefficients are constant (by Theorem \ref{symmodd} and
\cite[section 24]{GMS}). Over a field $k$ of characteristic not 2,
$\Inv_k^*(S_n,\F_2)$ is a free module over $H^*(k,\F_2)$ with basis
$1=w_0,w_1,\ldots,w_m$, where $m=\lfloor n/2\rfloor$
\cite[Theorem 25.13]{GMS}. Here the elements
$w_i$ are the Stiefel-Whitney classes of the trace form $\tr(xy)$
associated to an \etale algebra. For $k$ of characteristic 2, Theorem
\ref{symm2} says that there is only an analog of $w_1$.

\begin{theorem}
\label{symmodd}
Let $k$ be a field of characteristic $p>2$, and let $n$ be a positive
integer. For each integer $r$, every invariant of the symmetric
group $S_n$ over $k$ with values in mod $p$ \etale motivic
cohomology ($H^{r,r}$ or $H^{r+1,r}$) is constant.
\end{theorem}

\begin{proof}
For $H^{r,r}$, this follows from Theorem \ref{smoothmilnor}.
So let $u$ be an invariant for $G=S_n$ over $k$ with values
in $H^{r+1,r}$. Let $V$ be the standard
representation of $G$, of dimension $n$ over $k$. Then $u$ gives an element
of $H^{r+1,r}(k(V/G))$, and $u$ is determined by this element,
by Theorem \ref{bm}.

The action of $G$ of $V$ extends to the
permutation action of $G$ on $X=(\P^1)^n$ over $k$, with $X/G\cong \P^n$.
The group $G$ acts freely on $X$ outside the union of the
$\binom{n}{2}$ irreducible divisors $x_i=x_j$ in $X$,
where $1\leq i<j\leq n$.
These divisors are permuted transitively by $G$, and so the morphism
$X\arrow X/G=\P^n$ is ramified only over one irreducible divisor,
the discriminant $\Delta\subset\P^n$.

By Theorem \ref{bm}, using that $X$ is a compactification
of a representation of $G$, the element $u\in H^{r+1,r}(k(X/G))$
is unramified outside the divisor $\Delta$.
Likewise, the alternating group $A_n$ acts freely on $X=(\P^1)^n$
outside a closed subset of codimension at least 2, and so the pullback
of $u$ to $H^{r+1,r}(k(X/A_n))$ is unramified along every
irreducible divisor in $X/A_n$.

Since $p$ is odd and the class $u$ pulls back to an unramified class
by the double cover $X/A_n\arrow X/S_n$, $u$ is in fact unramified
along every irreducible divisor in $X/S_n\cong \P^n$. (This follows
from the description of $H^{r+1,r}(k(X/S_n))/H^{r+1,r}_{\nr}(k(X/S_n))$
in Theorem \ref{filtration}, where ``nr'' denotes
the subgroup of classes unramified along $\Delta$. Use
that a uniformizer $t$ (the discriminant polynomial)
in $k(X/S_n)$ along $\Delta$ pulls back in $k(X/A_n)$
to $u^2$, for some uniformizer $u$
along the inverse image of $\Delta$.)

By Theorem \ref{k(t)}, every unramified cohomology class in $H^{r+1,r}$
of the function field of $\P^1$ over a field $k_0$ is pulled back
from a unique class on $k_0$. Applying this repeatedly gives the same
statement on the function field of $\P^n$. It follows that the class
$u$ is pulled back from $H^{r+1,r}(k)$. Thus $u$ is constant
as an invariant of $G$.
\end{proof}

\begin{theorem}
\label{symm2}
Let $k$ be a field of characteristic $2$, and let $n\geq 2$.
For each integer $r$, every invariant of the symmetric
group $S_n$ over $k$ with values in $H^{r,r}$ is constant.
Also, the group of invariants with values in $H^{r+1,r}$
is $H^{r+1,r}(k)\oplus H^{r,r}(k)$. Every invariant for $S_n$ over $k$
in $H^{r+1,r}$ has the form
$$u(x)=c+\disc(x)e$$
for some (unique) $c\in H^{r+1,r}(k)$ and $e\in H^{r,r}(k)$.
Here $\disc(x)$ is the invariant of $S_n$ in $H^{1,0}$
corresponding to the sign homomorphism $S_n\arrow \Z/2$.
\end{theorem}

\begin{proof}
Every invariant of $S_n$ with values in $H^{r,r}$ is constant
by Theorem \ref{smoothmilnor}. We now consider invariants
in $H^{r+1,r}$. I claim that the restriction
$$\NormInv_k^{r+1,r}(S_n)\arrow \NormInv_k^{r+1,r}(S_2\times S_{n-2})$$
is injective. Indeed, let $u$ be a normalized invariant
for $S_n$ that restricts to 0 as an invariant
of $S_2\times S_{n-2}$. As in the proof of Theorem \ref{symmodd},
consider the action of $G=S_n$ on $X=(\P^1)^n$. We know
that $u$ is determined by its class in $H^{r+1,r}(k(X/S_n))$,
and that this class is unramified outside the discriminant
divisor $\Delta$ in $X/S_n\cong \P^n$.

We are given that $u$ pulls back to 0
in $H^{r+1,r}(k(X/(S_2\times S_{n-2})))$. The point is that $S_n$ acts
transitively on the set of divisors $x_i=x_j$ in $X=(\P^1)^n$,
and the stabilizer subgroup of the divisor $x_1=x_2$
is $S_2\times S_{n-2}$. As a result, the map
$X/(S_2\times S_{n-2})\arrow X/S_n$ splits completely over $\Delta$;
that is, the completions of the two function fields along the corresponding
divisors are isomorphic. It follows that $u\in H^{r+1,r}(k(X/S_n))$
is unramified along $\Delta$. Since $u$ is also unramified along all
other irreducible divisors in $\P^n_k$, $u$ is pulled back
from $H^{r+1,r}(k)$. Since $u$ pulls back to 0 as an invariant
of $S_2\times S_{n-2}$, $u$ is equal to 0 in $H^{r+1,r}(k)$,
as we want.

By Proposition \ref{zpprod}, we have $\NormInv_k^{r+1,r}(S_2
\times S_{n-2})\cong \NormInv_k^{r+1,r}(S_{n-2})\oplus \Inv_k^{r,r}
(S_{n-2})$. Since $S_{n-2}$ is smooth over $k$, $\Inv_k^{r,r}(S_{n-2})$
is isomorphic to $H^{r,r}(k)$ by Theorem \ref{smoothmilnor}.
So $\NormInv_k^{r+1,r}(S_2\times S_{n-2})
\cong \NormInv_k^{r+1,r}(S_{n-2})\oplus H^{r,r}(k)$.
Let $m=\lfloor n/2\rfloor$. Repeatedly applying the isomorphism just mentioned
together with the previous paragraph's result, we find that restricting
from $S_n$ to its subgroup $(\Z/2)^m$ gives an injection
$$\varphi\colon \NormInv_k^{r+1,r}(S_n)\inj \oplus_{i=1}^m H^{r,r}(k).$$
Since the normalizer of $(\Z/2)^m$ in $S_n$ contains $S_m$,
the image of $\varphi$ must be fixed by $S_m$. So we have an injection
$$\NormInv_k^{r+1,r}(S_n)\inj H^{r,r}(k).$$
That is, every normalized invariant $u$ of $S_n$
is determined by its restriction to the subgroup
$H=\langle (12)\rangle\cong \Z/2\subset S_n$, where it has the form
$u([a])=[a]e$ for some $e\in H^{r,r}(k)$, writing $[a]$ for an element
of $H^{1,0}$.

Conversely, for any $e\in H^{r,r}(k)$,
there is a normalized invariant of $S_n$ that restricts
to the invariant $u([a])=[a]e$ on the subgroup $H$;
namely, the pullback of $e\in \NormInv_k^{r+1,r}(\Z/2)\cong H^{r,r}(k)$
via the sign homomorphism $S_n\arrow \Z/2$. (Here we use
that the composition $\langle (12)\rangle\subset S_n\arrow \Z/2$
is the identity.) Thus we have shown that $\NormInv_k^{r+1,r}(S_n)
\cong H^{r,r}(k)$.
\end{proof}

\section{Invariants with values in $H^{n,n}$}

Although much of this paper is on invariants of affine
group schemes in $H^{n+1,n}$, in this section we prove analogous
results for invariants in $H^{n,n}$, the other part of mod $p$ \etale
motivic cohomology for fields. Specifically, we compute the invariants
for $\mu_p\times H$ and $\Z/p\times H$ for any group scheme $H$.

First, Vial computed the invariants of $\mu_p$ with values
in $H^{n,n}$, as mentioned in Theorem \ref{vialops}:
for any field $k$ of characteristic $p>0$,
$$\Inv_k^{n,n}(\mu_p)\cong H^{n,n}(k)\oplus H^{n-1,n-1}(k).$$
In more detail, for $a\in H^{n,n}(k)$ and $b\in H^{n-1,n-1}(k)$,
the corresponding invariant of a $\mu_p$-torsor $\xi$
over a field $F/k$ is $a+b\xi$, where $\xi\in H^1(F,\mu_p)=H^{1,1}(F)$.

From there, we can compute the invariants of $\mu_p\times H$ in $H^{n,n}$
for any group $H$.

\begin{proposition}
\label{mupprodnn}
Let $H$ be an affine group scheme of finite type over a field $k$.
Then
$$\Inv^{n,n}_k(\mu_p\times H)\cong \Inv_k^{n,n}(H)\oplus
\Inv_k^{n-1,n-1}(H).$$
The isomorphism sends invariants $a$ and $b$ for $H$ (with $a$ of 
bidegree $(n,n)$ and $b$ of bidegree $(n-1,n-1)$) to 
$a+b\xi$, where $\xi$ is the obvious invariant
for $\mu_p$ of bidegree $(1,1)$, coming from the isomorphism
$H^1(F,\mu_p)\cong H^{1,1}(F)$ for fields $F$ over $k$.
\end{proposition}

\begin{proof}
We follow the proof of Proposition \ref{mupprod} almost verbatim.
Let $V$ be a $k$-vector space on which $H$ acts by affine
transformations, and suppose that $H$ acts freely
on a nonempty open subset $U$ of $V$ and the quotient
scheme $U/H$ exists. (Such pairs $(V,U)$ do exist
\cite[Remark 2.7]{Totarobook}.)

Let $u$ be an invariant of $\mu_p\times H$ over $k$ with values
in $H^{n,n}$. For any field $L$ over $k$ and any $H$-torsor
$\beta$ over $L$, we get an invariant $u_{\beta}$ of $\mu_p$
over $L$ with values in $H^{n,n}$ by defining
$$u_{\beta}(\alpha)=u(\alpha,\beta)$$
for any $\mu_p$-torsor $\alpha$ over an extension field
of $L$. By Vial's result above, there are unique elements
$v\in H^{n,n}(L)$ and $w\in H^{n-1,n-1}(L)$ such that
$u_{\beta}(\alpha)=v+w\alpha$ for all $\mu_p$-torsors
$\alpha$ over fields over $L$. Here we are identifying
$H^1(E,\mu_p)$ with $H^{1,1}(E)$, for fields $E$ over $L$.

By that uniqueness, $v$ and $w$ are invariants of $H$-torsors
$\alpha$ on fields over $k$. These invariants satisfy (and are
characterized uniquely by): for every $(\mu_p\times H)$-torsor
$(\alpha,\beta)$ on a field $E$ over $k$,
$$u(\alpha,\beta)=v(\beta)+w(\beta)\alpha.$$
Thus every invariant for $\mu_p\times H$ has this form,
with the invariants $v$ and $w$ uniquely determined.
Conversely, for any invariants $v$ and $w$ for $H$ over $k$,
the formula above defines an invariant for $\mu_p\times H$.
Thus we have shown that 
$$\Inv^{n,n}_k(\mu_p\times H)\cong \Inv_k^{n,n}(H)\oplus
\Inv_k^{n-1,n-1}(H).$$
\end{proof}

We now compute the invariants of $\Z/p\times H$ with values
in $H^{n,n}$. More generally, we can handle $G\times H$ for
any smooth $k$-group $G$.

\begin{proposition}
\label{smoothprod}
Let $k$ be a field of characteristic $p>0$.
Let $G$ and $H$ be affine $k$-group schemes of finite type
over $k$ with $G$ smooth over $k$. Then
$$\Inv_k^{n,n}(G\times H)\cong \Inv_k^{n,n}(H).$$
\end{proposition}

The proof is identical to that of Proposition \ref{mupprodnn},
starting from the fact that $\Inv_F^{n,n}(G)\cong H^{n,n}(F)$
for every field $F$ over $k$ (Theorem \ref{smoothmilnor}).

\section{Operations on \etale motivic cohomology of fields}
\label{operationsect}

Vial found all operations on Milnor $K$-theory mod $l$
of fields over a given field $k$ \cite[Theorem 1]{Vial}. Roughly speaking,
all operations are spanned by Kahn and Rost's divided power operations.
(By contrast, Steenrod operations are essentially trivial
on the motivic cohomology of fields.)
Here $l$ may be equal to the characteristic of $k$, and so Vial's
result describes all operations on the mod $p$ \etale motivic cohomology
groups $H^{n,n}$ of fields of characteristic $p$.

We now find all operations on the mod $p$ \etale
motivic cohomology groups, both $H^{m,m}$ and $H^{m+1,m}$, in
characteristic $p$. (Think of $H^{m,m}$ or $H^{m+1,m}$ as a functor
from fields over $k$ to sets; then an ``operation'' means a natural
transformation from one such functor to another. In particular, operations
are not assumed to be additive.) In short, only the known operations
exist. The proofs use the computation of the cohomological invariants
of the group scheme $(\Z/p)^r\times (\mu_p)^s$ (Theorem \ref{abgp}).

We state four theorems, describing operations from $H^{m,m}$ or $H^{m+1,m}$
to $H^{r,r}$ or $H^{r+1,r}$. First, here is Vial's theorem on operations
from $H^{m,m}$ to $H^{r,r}$, in the case of mod $p$
cohomology for fields of characteristic $p$.
If $p=2$ and $m\geq 2$, or if $p$ is odd
and $m\geq 2$ is even, then (by Kahn and Rost)
there are {\it divided power }operations
$\gamma_i\colon H^{m,m}(F)\arrow H^{im,im}(F)$ for all $i\geq 0$
and all fields $F$ of characteristic $p$, defined
on a sum of symbols $s_j=\{b_{j1},\ldots,b_{jm}\}$ by
$$\gamma_i\bigg( \sum_{j=1}^n s_j\bigg) =\sum_{|T|=i} \prod_{j\in T}s_j,$$
where the sum runs over all subsets $T$ of $\{1,\ldots,n\}$
of order $i$. These are typically not additive operations; instead,
they satisfy $\gamma_i(x+y)=\sum_{j=0}^i \gamma_j(x)\gamma_{i-j}(y)$
\cite[Properties 2.3]{Vial}.

\begin{theorem}
\label{vialthm}
(Vial)
For a field $k$ of characteristic $p>0$, the group of operations
$H^{m,m}\arrow H^{r,r}$ on fields over $k$ is of the form:

(1) if $m=0$: $H^{r,r}(k)^{\oplus p}$;

(2) if $p=2$ and $m=1$, or $p$ is odd and $m\geq 1$ is odd:
$H^{r,r}(k)\oplus H^{r-m,r-m}(k)$, with every operation
of the form $u(x)=c+ex$ for some (unique) $c$ and $e$;

(3) if $p=2$ and $m\geq 2$, or $p$ is odd and $m\geq 2$ is even:
every operation has the form $u(x)=\sum_{i\geq 0}c_i\gamma_i(x)$
for some (unique) elements $c_i$ in $H^{r-im,r-im}(k)$.
\end{theorem}

We now state the other three theorems on operations.

\begin{theorem}
\label{Milnortoetale}
For a field $k$ of characteristic $p>0$, the group of operations
$H^{m,m}\arrow H^{r+1,r}$ on fields over $k$ is as listed
in Theorem \ref{vialthm}, but with the coefficients $c$, $e$,
and so on in $H^{j+1,j}$ rather than $H^{j,j}$.
\end{theorem}

\begin{theorem}
\label{etaletoMilnor}
For a field $k$ of characteristic $p>0$, every operation
$H^{m+1,m}\arrow H^{r,r}$ on fields over $k$ is constant.
In particular, every normalized operation is zero.
\end{theorem}

\begin{theorem}
\label{operations}
For a field $k$ of characteristic $p>0$ and a natural number $m$,
every operation
$H^{m+1,m}\arrow H^{r+1,r}$ on fields over $k$ is of the form
$u(x)=c+ex$ for some (unique) elements
$c\in H^{r+1,r}(k)$ and $e\in H^{r-m,r-m}(k)$.
In particular, every normalized operation is additive.
\end{theorem}

\begin{proof} (Theorem \ref{Milnortoetale})
In the notation of section \ref{background},
every element of $H^{m+1,m}(F)$ (for a field $F$ over $k$) can be written
as a finite sum of symbols
$$x=\sum_{i=1}^n [a_i,b_{i1},\ldots,b_{im}\}$$
with $a_i\in F$ and $b_{ij}\in F^*$.
Moreover, this expression in $H^{m+1,m}(F)$ only depends on the
classes of $a_i$ in $F/\mathcal{P}(F)=H^{1,0}(F)$ and $b_{ij}$
in $(F^*)/(F^*)^p=H^{1,1}(F)$. Here $\mathcal{P}(a)=a^p-a$.

Let $u$ be an operation from $H^{m,m}$ to $H^{r+1,r}$ on fields
over $k$. The case $m=0$ is easy, since $H^{0,0}(F)\cong \F_p$
for every field $F$ over $k$. So assume that $m$ is positive.
If $m=1$, then an operation from $H^{1,1}$ to $H^{r+1,r}$ is the same
as an invariant of the group scheme $\mu_p$ over $k$ with
values in $H^{r+1.r}$, and these are described in Proposition
\ref{mup}. So we can assume that $m$ is at least 2.

Let $n$ be a positive integer, and write
$\vec{n}$ for the set $\{1,\ldots,n\}$. Applying the operation
$u$ to sums of
$n$ symbols gives an invariant of the group scheme $(\mu_p)^{mn}$
over $k$ with values in $H^{r+1,r}$. By Theorem \ref{abgp}, this has the form,
for $x=\sum_{i=1}^n\{b_{i1},\ldots,b_{im}\}$:
$$u(x) =\sum_{T\subset
\vec{n}\times\vec{m}}c_T
\prod_{(i,j)\in T}\{b_{ij}\}$$
for some (unique) elements $c_T\in H^{r-|T|+1,r-|T|}(k)$.

If $b_{ij}=1\in k^*$ for some pair $(i,j)$, then $\{b_{i1},\ldots,b_{im}\}=0$,
and so the operation above must be independent of $b_{il}$ for all
$l\neq j$. By the uniqueness in Theorem \ref{abgp}, it follows that
$u$ must have the form:
$$u\bigg( \sum_{i=1}^n\{b_{i1},\ldots,b_{im}\}\bigg) =\sum_{T\subset
\vec{n}}c_T
\prod_{i\in T}\{b_{i1},\ldots,b_{im}\}.$$
Also, the operation must be independent of the order of the $n$ summands
in $x$. If $p=2$, or if $p>2$ and $m$
is even, then multiplication of elements of $H^{m,m}$ is commutative.
In that case, $u$ must have the form:
$$u\bigg( \sum_{i=1}^n\{b_{i1},\ldots,b_{im}\}\bigg) =\sum_{j=0}^n c_j
\sum_{\substack{T\subset \vec{n}\\|T|=j}}
\prod_{i\in T}\{b_{i1},\ldots,b_{im}\}.$$
Thus every operation is a linear combination (with coefficients
in $H^{*+1,*}$) of divided power operations. Conversely, divided
power operations are well-defined under our assumptions (that
$m\geq 2$ and, if $p$ is odd, then $m$ is even),
by Theorem \ref{vialthm}. Here we have considered
operations on elements of $H^{m,m}$ written as a sum of a fixed number
of symbols, but (since we can take one symbol to be zero) these descriptions
must be compatible as the number of symbols varies. This completes
the proof under the assumptions mentioned.

There remains the case where $p>2$ and $m$ is odd. Here multiplication
of elements of $H^{m,m}$ is anti-commutative. In this case, since the operation
$u(x)$ must be unchanged
after switching two summands of $T$, we must have (in the notation above)
$c_T=-c_T$ for every set $T\subset \{1,\ldots,n\}$ of order at least 2.
Since $c_T$ is an element of an
$\F_p$-vector space with $p$ odd, that means that $c_T=0$ if $T$
has order at least 2. So $u$ has the form
$$u\bigg( \sum_{i=1}^n\{b_{i1},\ldots,b_{im}\}\bigg) =c+\sum_{i=1}^n e_i
\{b_{i1},\ldots,b_{im}\}.$$
Using again that $u$ is unchanged by permuting the summands,
we have $e_1=\cdots=e_n$. So $u$ has the form
$$u(x)=c+ex$$
for some $c,e\in H^{*+1.*}(k)$.
\end{proof}

\begin{proof}
(Theorem \ref{etaletoMilnor})
Let $u$ be an operation from $H^{m+1,m}$ to $H^{r,r}$ on fields
over $k$. Applying $u$ to sums of $n$ symbols,
$$u\bigg( \sum_{i=1}^n [a_i,b_{i1},\ldots,b_{im}\}\bigg) $$
gives an invariant
of the group scheme $(\Z/p)^n\times (\mu_p)^{mn}$ over $k$
with values in $H^{r,r}$. By Proposition \ref{smoothprod},
such an invariant must be independent of $a_1,\ldots,a_n\in H^{1,0}(k)$.
But if we take those elements to be zero, then the element
$\sum_{i=1}^n [a_i,b_{i1},\ldots,b_{im}\}$ in $H^{m+1,m}$
is zero. So every operation from $H^{m+1,m}$ to $H^{r,r}$
is constant.
\end{proof}

\begin{proof}
(Theorem \ref{operations})
Let $u$ be any operation from $H^{m+1,m}$ to $H^{r+1,r}$
on fields over $k$. For a positive integer $n$,
restricting $u$ to sums of $n$ symbols
gives an invariant of the group scheme $(\Z/p)^n\times (\mu_p)^{mn}$
over $k$ with values in $H^{r+1,r}$. By Theorem \ref{abgp},
we can write $u$ on an element $x=\sum_{i=1}^n [a_i,b_{i1},\ldots,b_{im}\}$
as
$$u(x)=
\sum_{T\subset \vec{n}\times \vec{m}}c_T\prod_{(i,j)\in T}\{b_{ij}\}
+\sum_{l=1}^n [a_l]\sum_{T\subset \vec{n}\times \vec{m}}e_{l,T}
\prod_{(i,j)\in T}\{b_{ij}\}$$
for some (unique) elements
$c_T$ in $H^{r-|T|+1,r-|T|}(k)$ and $e_{l,T}$ in $H^{r-|T|,r-|T|}(k)$.

In fact, all coefficients $c_T$ with $T$ nonempty are zero, because
the input $x$ in $H^{r+1,r}$ is zero if all $a_i$ are zero,
no matter what the $b_{i,j}$ are.
Thus $u$ can be written as:
$$u\bigg( \sum_{i=1}^n [a_i,b_{i1},\ldots,b_{im}\}\bigg) =
c+\sum_{l=1}^n [a_l]\sum_{T\subset \vec{n}\times\vec{m}}e_{l,T}
\prod_{(i,j)\in T}\{b_{ij}\}.$$

Next, let $1\leq i\leq m$. Note that the term
$[a_i,b_{i1},\ldots,b_{im}\}$ in $x$ is zero if $a_i$ is zero
or if any of $b_{i1},\ldots,b_{im}$ is 1.
So, if $a_i$ is 0, then $u(x)$ must be independent of
$b_{i1},\ldots,b_{im}$; and if some $b_{ij}$ is equal to 1,
then $u(x)$ must be independent of $a_i$. Using the uniqueness
of the coefficients (from Theorem \ref{abgp}) again, it follows
that $e_{s,T}$ is zero for all $T\neq \{s\}\times \vec{m}$.
That is, $u$ can be written as:
$$u\bigg( \sum_{i=1}^n [a_i,b_{i1},\ldots,b_{im}\}\bigg) =
c+\sum_{i=1}^n
[a_i,b_{i1},\ldots,b_{im}\}e_i$$
for some elements $c$ in $H^{r+1,r}(k)$ and $e_i$
in $H^{r-m,r-m}(k)$. 

Finally, the operation $u$ must be unchanged if we permute
the $n$ summands in the input. It follows that $e_1=\cdots=e_n$.
That is, the operation $u$ is given on sums of $n$ symbols by
$$u(x)=c+xe$$
for some (uniquely determined) $c$ in $H^{r+1,r}(k)$
and $e$ in $H^{r-m,r-m}(k)$. Since we can take one symbol to be zero,
these elements $c$ and $e$
must be unchanged if we change the number $n$ of symbols in $x$.
That is,
the operation $u$ is given by $u(x)=c+xe$ on all of $H^{m+1,m}(F)$,
for fields $F$ over $k$.
\end{proof}

\section{Invariants of the even orthogonal group in characteristic 2}
\label{orthosect}

Define
a quadratic form $q_0$ on a vector space $V$ over a field $k$ to be
{\it nonsingular }if the orthogonal complement $V^{\perp}\subset V$
has dimension at most 1 and $q_0$ is nonzero at each nonzero element of
$V^{\perp}$. {\it Quadratic forms
will be understood to be nonsingular in this paper. }One reason
for the importance of this class
of quadratic forms is that the simple algebraic groups
of type $B_n$ and $D_n$ over any field are essentially
automorphism groups of nonsingular
quadratic forms. Note that
if $k$ has characteristic 2, then the bilinear form
$b_0(x,y)=q_0(x+y)-q_0(x)-q_0(y)$ associated
to $q_0$ is alternating. So $V^{\perp}$ has dimension 0 if $q_0$
has even dimension and dimension 1 if $q_0$ has odd dimension.

Let $q_0$ be a quadratic form of even dimension
over a field $k$ of characteristic 2. In Theorems
\ref{evenorthog} and \ref{oplus}, we compute
the cohomological invariants for the orthogonal
group $O(q_0)$ and its identity
component, which we call $SO(q_0)$ (even though $O(2n)$
is contained in $SL(2n)$ in characteristic 2).
We consider
the invariants for values in $H^{m+1,m}$; since these group schemes
are smooth, their invariants in $H^{m,m}$ are constant
by Theorem \ref{smoothmilnor}. In short, the fundamental
invariants are the discriminant (or Arf invariant)
and the Clifford invariant, which are described in the proof
of Theorem \ref{evenorthog}.

\begin{theorem}
\label{evenorthog}
Let $k$ be a field of characteristic 2,
$m$ an integer. Let $q_0$ be a quadratic form of dimension $2n$
over $k$ with $n\geq 1$.
Then
$$\Inv_k^{m+1,m}(O(q_0))\cong H^{m+1,m}(k)\oplus H^{m,m}(k)\oplus
H^{m-1,m-1}(k).$$
Explicitly, we can view the invariants for $O(q_0)$
as the invariants of
quadratic forms of dimension $2n$ over fields
$F/k$. Every invariant has the form
$$u(q)=c+\disc(q)e+\clif(q)f$$
for some (uniquely determined) $c\in H^{m+1,m}(k)$,
$e\in H^{m,m}(k)$, and $f\in H^{m-1,m-1}(k)$.
\end{theorem}

Note the contrast with Serre's calculation in characteristic not 2:
for a quadratic form $q_0$ of dimension $m$
over a field $k$ of characteristic
not 2, 
$\Inv_k^*(O(q_0),\F_2)$ is a free module over $H^*(k,\F_2)$ with basis
the Stiefel-Whitney classes $1=w_0,w_1,w_2,
\ldots,w_m$ \cite[Theorem 17.3]{GMS}. (In characteristic not 2,
the weight makes no difference; that is, the \etale motivic cohomology
group $H^i_{\et}(F,\Z/2(j))$ is the same for all $j\geq 0$.)

In characteristic 2, Theorem \ref{evenorthog} says that
there are analogs of $w_1$ (the discriminant
or Arf invariant in $H^{1,0}$) and $w_2$ (the Clifford invariant
in $H^{2,1}$), but no analogs of the higher Stiefel-Whitney classes.
This is a bit disappointing, but note that even in characteristic not 2,
the first two Stiefel-Whitney classes of a quadratic form are far
more important than the higher ones. For example, if $F$ is a field
of characteristic not 2 in which $-1$ is a square, then $w_1$ and $w_2$
give isomorphisms $w_1\colon I/I^2\arrow H^1(F,\Z/2)$ 
and $w_2\colon I^2/I^3\arrow H^2(F,\Z/2)$, but all Stiefel-Whitney classes
of positive degree vanish on $I^3$ \cite[Exercise 5.14]{EKM}.
Thus, for $j\geq 3$, the isomorphism
$I^j/I^{j+1}\cong H^j(F,\Z/2)$ proved by Orlov-Vishik-Voevodsky \cite{OVV}
does not come from invariants
defined on all quadratic forms of a given dimension, but only
from invariants on some subclass of forms.

This line of thought suggests looking at the invariants
of the connected group $SO(q)$
and its double cover $\Spin(q)$ in characteristic 2.
In this paper, we only find the invariants for $SO(q)$. 
We know that $\Spin(q)$ will have a nontrivial
invariant in $H^{3,2}$ by Kato's isomorphism
$$I^{n+1}_q(k)/I^{n+2}_q(k)\cong H^{n+1,n}(k),$$
applied in the case $n=2$
\cite{Katoquadratic}. (We use the notation of \cite[section 9.B]{EKM}:
$I_q(k)$ is the quadratic Witt group, which is a module over the bilinear
Witt ring $W(k)$, and $I^n_q(k):=I^{n-1}I_q(k)$ for $n\geq 1$. For
the hyperbolic form $q_0=nH$, torsors for $\Spin(q_0)$ over $k$
give quadratic
forms in $I^3_q(k)$.)
This invariant for $\Spin(q)$ was
generalized by Merkurjev to the Rost invariant
of any simply connected group \cite[Part 2, Theorem 9.11]{GMS}.
For $n\leq 14$, some higher-degree invariants
of $\Spin(n)$ have been constructed by Rost and Garibaldi
in characteristic not 2 and by the author in characteristic 2
\cite[section 23]{Garibaldi}, \cite{Totarospin}.
It would be interesting to construct
invariants for spin groups of higher dimensions.

\begin{proof}
(Theorem \ref{evenorthog})
For any field $F$ over $k$, $H^1(F,O(q_0))$ can be identified
with the set of isomorphism classes of quadratic
forms over $F$ of dimension $2n$ \cite[equation 29.28]{KMRT}.
So computing the invariants
for $O(q_0)$ amounts to computing the invariants for quadratic
forms of dimension $2n$. In particular, this description
shows that the invariants of $O(q_0)$ 
are the same for all quadratic
forms $q_0$ of dimension $2n$ over $k$. So we can assume that
$q_0$ is the simplest quadratic form, $q_0=nH$, the orthogonal
direct sum of $n$ copies of the hyperbolic plane $q_H(x,y)=xy$.

The group scheme $\Z/2$ is contained in $O(H)$ by switching $x$ and $y$,
and this commutes with the action of the group scheme
$\mu_2$ by scalar multiplication. So we have a subgroup
$\Z/2\times \mu_2$ in $O(H)$, and hence a subgroup $(\Z/2)^n\times (\mu_2)^n$
in $O(nH)$. Let $F$ be a field over $k$.
For elements $a\in F$ and $b\in F^*$, which give
a $\Z/2$-torsor $[a]$ and a $\mu_2$-torsor $(b)$ over $F$,
the associated 2-dimensional quadratic form
(given by $H^1(F,\Z/2\times\mu_2)\arrow H^1(F,O(H))$
can be written as $b\langle\langle a]]=b[1,a]=bx^2+bxy+aby^2$.
Every quadratic
form of dimension 2 over $F$ arises this way; that is, every
form of dimension 2
is a scalar multiple of a 1-fold Pfister form \cite[section 9.B]{EKM}.
Moreover, every quadratic form over $F$ of dimension $2n$ is an orthogonal
direct sum of 2-dimensional forms \cite[Corollary 7.3.2]{EKM}, and so
$$H^1(F,(\Z/2)^n\times (\mu_2)^n)\arrow H^1(F, O(nH))$$
is surjective.

As a result, for the quadratic form $q_0=nH$, the restriction
$$\Inv^{m+1,m}_k(O(q_0))\arrow \Inv^{m+1,m}_k((\Z/2)^n\times (\mu_2)^n)$$
is injective. By Theorem \ref{abgp}, every invariant for $O(q_0)$
over $k$ with values in $H^{m+1,m}$ can be written as:
$$u\bigg( \sum_{i=1}^n b_i\langle\langle a_i]]\bigg) =
\sum_{I\subset \{1,\ldots,n\}}c_I\prod_{i\in I}\{b_i\}
+\sum_{j=1}^n [a_j]\sum_{I\subset \{1,\ldots,n\}}e_{j,I}\prod_{i\in I}\{b_i\}
$$
for some (uniquely determined) $c_I\in H^{m-|I|+1,m-|I|}(k)$ and
$e_{j,I}\in H^{m-|I|,m-|I|}(k)$.

If $a_1=\cdots=a_n=0$, then the quadratic form
$\sum_i b_i\langle\langle a_i]]$ is hyperbolic. So the invariant above
is constant (independent of $b_1,\ldots,b_n\in k^*$) in that case.
By the uniqueness in Theorem \ref{abgp}, it follows that
$c_I=0$ for all $I\neq \emptyset$.

Next, if $a_j=0$, then the quadratic form
$b_j\langle\langle a_j]]$ is hyperbolic, and so
the invariant above is independent of $b_j\in k^*$. So $e_{l,I}=0$
unless $I$ is empty or $I=\{l\}$. Thus the invariant has the form
$$u\bigg( \sum_{i=1}^n b_i\langle\langle a_i]]\bigg) =
c+\sum_{j=1}^n [a_j]e_j+\sum_{j=1}^n [a_j,b_j\}f_j,$$
for some (uniquely determined) $c\in H^{m+1,m}(k)$,
$e_j\in H^{m,m}(k)$, and $f_j\in H^{m-1,m-1}(k)$.

The invariant $u$ must be invariant under permuting the $n$ pairs
$(a_1,b_1),\ldots,(a_n,b_n)$. It follows that $e_1=\cdots=e_n$
and $f_1=\cdots=f_n$. That is,
$$u\bigg( \sum_{i=1}^n b_i\langle\langle a_i]]\bigg) =
c+\bigg[\sum_{j=1}^n [a_j]\bigg] e+\bigg[\sum_{j=1}^n [a_j,b_j\}
\bigg]f$$
for some (uniquely determined) $c\in H^{m+1,m}(k)$,
$e\in H^{m,m}(k)$, and $f\in H^{m-1,m-1}(k)$. 

The {\it discriminant }(or {\it Arf invariant}) $\disc(q)$ of the quadratic
form $q=\sum_{i=1}^n b_i\langle\langle a_i]]$ is $\sum_{j=1}^n a_j\in
k/\mathcal{P}(k)=H^{1,0}(k)$ \cite[Example 13.5]{EKM}.
Also, the {\it Clifford invariant }$\clif(q)$ is
$\sum_{j=1}^n [a_j,b_j\} \in H^{2,1}(k)=\Br(k)[2]$
\cite[section 14]{EKM}. Since these
are known to be invariants of quadratic forms,
we have determined all the invariants
for $O(q_0)$.
\end{proof}

\section{Invariants of $O(2n+1)$ and $SO(2n+1)$}

Let $q_0$ be a quadratic form on a vector space $V$
of dimension $2n+1$ over a field $k$ of characteristic 2.
(Quadratic forms are understood to be nonsingular
in the sense of section \ref{orthosect}.)
Then the orthogonal group $O(q_0)$ is not smooth
over $k$; it is a product $\mu_2\times SO(q_0)$, with
$SO(q_0)$ smooth and connected over $k$.
In this section, we determine the invariants
for both $O(q_0)$ and $SO(q_0)$. Note a difference
between even- and odd-dimensional quadratic forms in characteristic 2:
the discriminant of an odd-dimensional quadratic form lies
in $H^{1,1}(k)=H^1(k,\mu_2)$, not in $H^{1,0}(k)=H^1(k,\Z/2)$.

\begin{theorem}
\label{oddorthog}
Let $k$ be a field of characteristic 2, $n$ a positive integer.
Let $q_0$ be a quadratic form
of dimension $2n+1$ over $k$. For any integer $m$,
$$\Inv_k^{m+1,m}(O(q_0))\cong H^{m+1,m}(k)\oplus H^{m,m-1}(k)\oplus
H^{m-1,m-1}(k)\oplus H^{m-2,m-2}(k).$$
Explicitly, we can view the invariants of $O(q_0)$
as the invariants for
quadratic forms $q$ of dimension $2n+1$ over fields
$F/k$. Every invariant has the form
$$u(q)=c+\disc(q)e+\clif(q)f+\clif(q)\disc(q)g$$
for some (uniquely determined) $c\in H^{m+1,m}(k)$,
$e\in H^{m,m-1}(k)$, $f\in H^{m-1,m-1}(k)$,
and $g\in H^{m-2,m-2}(k)$.
\end{theorem}

\begin{proof}
Regardless of the choice of form $q_0$, the $O(q_0)$-torsors
over a field $F/k$ can be identified (up to isomorphism)
with the quadratic forms of dimension $2n+1$ over $F$.
Every nonsingular quadratic form on a vector space $V$
of dimension $2n+1$ over a field $F/k$ can be written
as the orthogonal direct sum of the 1-dimensional form $V^{\perp}$,
described by an element of $H^1(F,\mu_2)$, and a nonsingular
form of dimension $2n$. For an element $b_0$ in $F^*$, we write
$\langle b_0\rangle$ for the 1-dimensional quadratic form
$q(x)=b_0x^2$. So we can write $q_0=\langle b_0\rangle+q_1$
for some $b_0$ in $k^*$ and some nonsingular quadratic form $q_1$
over $k$ of dimension $2n$.
(Here $q_1$ is not uniquely determined by $q_0$.) Since every quadratic
form of dimension $2n+1$ over a field $F/k$ can be similarly decomposed
as $\langle b\rangle+r$,
the map $H^1(F,O(q_1)\times \mu_2)\arrow H^1(F,O(q_0))$
is surjective.

It follows that the restriction
$$\Inv^{m+1,m}_k(O(q_0))\arrow \Inv^{m+1,m}_k(O(q_1)\times \mu_2)$$
is injective. By
Theorems \ref{mupprod} and \ref{evenorthog}, it follows that
every invariant for $O(q_0)$ has the form
$$u(\langle b\rangle+r)=c+\disc(r)e+\clif(r)f
+\{b\}g+\disc(r)\{b\}h+\clif(r)\{b\}l$$
for some (unique) $c\in H^{m-1,m}(k)$,
$e\in H^{m,m}(k)$, $f\in H^{m-1,m-1}(k)$,
$g\in H^{m,m-1}(k)$, $h\in H^{m-1,m-1}(k)$,
and $l\in H^{m-2,m-2}(k)$.

For a field $F$ over $k$ and any $b$ in $F^*$ and $a_1$ in $F$, the quadratic
form $\langle b\rangle+b\langle\langle a_1]]$
is isotropic, by inspection, and so it is isomorphic
to $\langle b\rangle+H$, where $H$ is the hyperbolic plane.
(This is a known failure of cancellation for quadratic
forms in characteristic 2 \cite[equation 8.7]{EKM}.)
So the given invariant $u$ must take the same value on
$\langle b\rangle+b\langle\langle a_1]]+(n-1)H$
as on $\langle b\rangle+nH$. That is,
$$c+\{b\}g=c+[a_1]e+[a_1,b\}f+\{b\}g+[a_1,b\}h$$
as invariants of $\mu_p\times \Z/p$ (where we used that
$\{b,b\}=0$ in $H^{2,2}$). By the description of the invariants
for $\mu_p\times \Z/p$ in Theorem \ref{abgp},
it follows that $e=0$ and $f=h$.
Thus the invariant $u$ has the form
$$u(\langle b\rangle+r)=c+(\clif(r)+\disc(r)\{b\})f
+\{b\}g+\clif(r)\{b\}l$$
for some (unique) $c\in H^{m-1,m}(k)$,
$f\in H^{m-1,m-1}(k)$, $g\in H^{m,m-1}(k)$,
and $l\in H^{m-2,m-2}(k)$.

Here $\{b\}$ in $H^{1,1}$ is an invariant of $q=\langle b\rangle
+r$, the {\it
discriminant }$\disc(q)$ (called the ``half-discriminant''
in \cite[IV.3.1.3]{Knus}).
(It is clear that this is an invariant of $q$,
because it describes the restriction of $q$ to the 1-dimensional
subspace $V^{\perp}\subset V$.) The other known invariant of odd-dimensional
quadratic forms in characteristic 2
is the {\it Clifford invariant} in the Brauer group $H^{2,1}$,
given by \cite[Corollary IV.7.3.2]{Knus}:
\begin{align*}
\clif(\langle b\rangle +r)&=\clif(br)\\
&= \clif(r)+\disc(r)\{b\}.
\end{align*}
Since $\clif(r)\{b\}$ is equal to $\clif(q)\disc(q)$,
that is also an invariant of $q$. Thus we have found all the invariants
of $q$.
\end{proof}

Since $O(2n+1)$ is not a smooth group scheme, its invariants
in $H^{m,m}$ are not immediate from Theorem \ref{smoothmilnor},
but they are easy to compute:

\begin{proposition}
Let $k$ be a field of characteristic 2,
$n$ a positive integer,
$q_0$ a quadratic form
of dimension $2n+1$ over $k$.
Then
$$\Inv_k^{m,m}(O(q_0))\cong H^{m,m}(k)\oplus H^{m-1,m-1}(k)$$
for every integer $m$.
Explicitly, we can view the invariants for $O(q_0)$
as the invariants of quadratic forms $q$ of dimension $2n+1$
over fields $F/k$. Every invariant in $H^{m,m}$ has the form
$$u(q)=c+\disc(q)e$$
for some (uniquely determined) $c\in H^{m,m}(k)$ and
$e\in H^{m-1,m-1}(k)$.
\end{proposition}

\begin{proof}
By the same argument as in the proof of Theorem \ref{oddorthog},
the restriction
$$\Inv^{m,m}_k(O(q_0))\arrow \Inv^{m,m}_k(O(q_1)\times \mu_2)$$
is injective, where we write $q_0=\langle b_0\rangle+q_1$
for a nonsingular quadratic form $q_1$ (not unique) of dimension $2n$.
By Propositions \ref{smoothprod}
and \ref{mup}, we know the invariants for $O(q_1)\times \mu_2$.
So any invariant $u$ in $H^{m,m}$ for quadratic forms $q$ of dimension $2n+1$
can be written as
$$u(\langle b\rangle+r)=c+\{b\}e$$
for some (unique) $c\in H^{m,m}(k)$ and $e\in H^{m-1,m-1}(k)$.
Here $\{b\}=\disc(q)$ is an invariant of $q=\langle b\rangle+r$.
Thus we have found all the invariants in $H^{m,m}$ for quadratic forms
of dimension $2n+1$.
\end{proof}

Now we turn to the smooth connected group $SO(2n+1)$. Since
it is smooth, its invariants in $H^{m,m}$
are all constant (Theorem \ref{smoothmilnor}). Here are
its invariants in $H^{m+1,m}$.

\begin{theorem}
\label{oddplus}
Let $k$ be a field of characteristic 2,
$m$ an integer, $n$ a positive integer, $q_0$
a quadratic form of dimension $2n+1$ over $k$. Then the group
of cohomological invariants for $SO(q_0)$ is given by 
$$\Inv_k^{m+1,m}(SO(q_0))\cong 
H^{m+1,m}(k)\oplus H^{m-1,m-1}(k).$$
Concretely, writing $[d]=\disc(q_0)\in H^{1,1}(k)$,
we can view the invariants for $SO(q_0)$
as the invariants of
quadratic forms $q$ of dimension $2n+1$ and discriminant $[d]$.
Every invariant has the form
$$u(q)=c+\clif(q)f$$
for some (uniquely determined) $c\in H^{m+1,m}(k)$ and 
$f\in H^{m-1,m-1}(k)$.
\end{theorem}

\begin{proof}
Over any field $F/k$, the torsors for $SO(q_0)$ can be identified
(up to isomorphism) with the
quadratic forms $q$ of dimension $2n+1$ and discriminant $[d]$.
We have $\disc(aq)=\{a\}+\disc(q)$
in $H^{1,1}$,
and so these invariants are in fact independent of $[d]$.

So we can assume that $q_0$ has discriminant 
$1\in (k^*)/(k^*)^2\cong H^{1,1}(k)$. Then $q_0$ can be written
as $\langle 1\rangle+q_1$ for some nonsingular quadratic form $q_1$
over $k$ of dimension $2n$. The inclusion $O(q_1)\subset SO(q_0)$
gives a surjection $H^1(F,O(q_1))\arrow H^1(F,SO(q_0))$, since
every form $q$ of dimension $2n+1$ with trivial discriminant over a field $F/k$
can be written as an orthogonal sum $\langle 1\rangle+r$
for some nonsingular quadratic form $r$ of dimension $2n$
(not unique).
So $\Inv^{m+1,m}_k(SO(q_0))$ injects into
$\Inv^{m+1,m}_k(O(q_1))$. By Theorem \ref{evenorthog}, every
invariant $u$ for $SO(q_0)$ can be written as
$$u(\langle 1\rangle+r)=c+\disc(r)e+\clif(r)f$$
for some (unique) $c\in H^{m+1,m}(k)$,
$e\in H^{m,m}(k)$, and $f\in H^{m-1,m-1}(k)$.

We use a special case of the isomorphism from the proof
of Theorem \ref{oddorthog}: for any field $F/k$
and $a_1\in F$, the quadratic form
$\langle 1\rangle +\langle\langle a_1]]
+(n-1)H$ is isomorphic to $\langle 1\rangle+nH$. So the invariant
$u(q)$ must take the same value on these two forms. That is,
$$c+[a_1]e=c,$$
and so $[a_1]e$ is equal to zero
as an invariant of $\Z/p$ (thinking of $a_1\in F$
as an element of $H^1(F,\Z/p)$ for fields $F/k$). By the description
of the invariants for $\Z/p$ (Proposition \ref{zp}),
it follows that $e=0$. Thus the invariant $u$ has the form
$$u(\langle 1\rangle+r)=c+\clif(r)f$$
for some (unique) $c\in H^{m+1,m}(k)$
and $f\in H^{m-1,m-1}(k)$.

Here $\clif(\langle 1\rangle+r)=\clif(r)$, by the description
of the Clifford invariant for odd-dimensional forms
in the proof of Theorem \ref{oddorthog}. So $\clif(r)$ is an invariant
of $q=\langle 1\rangle+r$. Thus we have found all the invariants
for $q$.
\end{proof}

\section{Invariants of the connected group
$SO(2n)$ in characteristic 2}

Let $k$ be a field of characteristic 2, $n$ a positive
integer, $q_0$ a quadratic form of dimension $2n$
over $k$. 
The orthogonal group $O(q_0)$ is smooth over $k$, with two connected
components. We write $SO(q_0)$ for the identity component,
even though the whole group $O(2n)$ is contained in $SL(2n)$
in characteristic 2. Since $SO(q_0)$ is smooth, its invariants
in $H^{m,m}$ are constant (Theorem \ref{smoothmilnor}).
Here are its invariants in $H^{m+1,m}$.

\begin{theorem}
\label{oplus}
Let $k$ be a field of characteristic 2,
$n$ a positive integer, $q_0$ a quadratic
form of dimension $2n$ over $k$. Let $[d]$ be the discriminant
of $q_0$ in $H^{1,0}(k)$. Then the group
of cohomological invariants for $SO(q_0)$ in $H^{m+1,m}$
for an integer $m$ is given by 
$$\begin{cases}
H^{m+1,m}(k)\oplus [d]H^{m-1,m-1}(k) &\text{if }n=1\\
H^{m+1,m}(k)\oplus H^{m-1,m-1}(k) \oplus \{\lambda\in H^{m-2,m-2}(k):
[d]\lambda=0\}& \text{if }n=2\\
H^{m+1,m}(k)\oplus H^{m-1,m-1}(k) & \text{if }n\geq 3.
\end{cases}$$
We can equivalently view the invariants for $SO(q_0)$
as the invariants of quadratic forms $q$ of dimension $2n$
and discriminant $[d]$ over fields
$F/k$. For $n\geq 3$, every invariant has the form
$$u(q)=c+\clif(q)f$$
for some (uniquely determined) $c\in H^{m+1,m}(k)$ and 
$f\in H^{m-1,m-1}(k)$. For $n=2$, a 4-dimensional quadratic form $q$
with discriminant $[d]$
has an invariant $b_{\lambda}(q)$ in $H^{m+1,m}$
for each $\lambda\in H^{m-2,m-2}(k)$ with $[d]\lambda=0$,
as well as the Clifford invariant in $H^{2,1}$.
\end{theorem}

The invariants for 4-dimensional quadratic forms with
given discriminant are analogous to those
found by Serre in all even dimensions at least 4 when the
characteristic is not 2 \cite[Proposition 20.1]{GMS}.
Likewise, the invariants for 2-dimensional quadratic forms with
given discriminant are analogous to those found by Serre
in dimension 2 when the characteristic is not 2
\cite[Exercise 20.9]{GMS}.

Every 1-dimensional torus over $k$
is of the form $SO(q_0)$ for some 2-dimensional quadratic form $q_0$,
and so Theorem \ref{oplus} describes all mod $p$ cohomological invariants
for every 1-dimensional torus. Blinstein and Merkurjev described
the cohomological invariants in degrees at most 3 for tori
of any dimension \cite[Theorem 4.3]{BM}.

\begin{proof}
The map $H^1(F,SO(q_0))\arrow H^1(F,O(q_0))$ is injective, with
image the set of isomorphism classes of $2n$-dimensional
quadratic forms over $F$ with discriminant $[d]$
\cite[equation 29.29]{KMRT}. So we can think of the invariants
for $SO(q_0)$ as the invariants for quadratic forms
(on fields over $k$) of dimension $2n$ with discriminant $[d]$.

Every such form $q$ over a field $F/k$ can be written as $q=r+b_1\langle\langle
\disc(r)+d]]$ for some quadratic form $r$ of dimension $2n-2$
and some $b_1\in F^*$. (Equivalently, for any subform $r_0$ of dimension
$2n-2$ in $q_0$, the subgroup $O(r_0)\times \mu_2\subset SO(q_0)$
induces a surjection
on $H^1$.) Assume that $n\geq 2$.
We know the invariants for $O(r_0)\times \mu_2$
by Theorems \ref{mupprod} and \ref{evenorthog}. So every invariant
$u$ in $H^{m+1,m}$ for $SO(q_0)$ can be written,
on a quadratic form $q=r+b_1\langle\langle\disc(r)+d]]$, as
$$u(q)=c+\disc(r)e+\clif(r)f
+\{ b_1\}g+\disc(r)\{b_1\}h+\clif(r)\{b_1\}\lambda$$
for some (unique) $c\in H^{m+1,m}(k)$, $e\in H^{m,m}(k)$,
$f\in H^{m-1,m-1}(k)$, $g\in H^{m,m-1}(k)$,
$h\in H^{m-1,m-1}(k)$, and $\lambda\in H^{m-2,m-2}(k)$.

We can apply this formula to
$r=s+b_2\langle\langle a_2]]$, for any quadratic form $s$ of dimension
$2n-4$ over a field $F/k$ and any $b_2\in F^*$.
This amounts to restricting the invariant $u$
to a subgroup $O(2n-4)\times (\Z/2\times \mu_2)\times \mu_2$.
We compute that for a quadratic
form $q=s+b_2\langle\langle a_2]]+b_1\ll a_2+\disc(s)+d]]$,
\begin{align*}
u(q)= & \; c+\disc(s)e+[a_2]e
+\clif(s)f+[a_2,b_2\}f+g\{b_1\}\\
& +\disc(s)\{b_1\}h+[a_2,b_1\}h+\clif(s)\{b_1\}\lambda+[a_2,b_1,b_2\}\lambda.
\end{align*}
This must be unchanged when we switch $b_1$ and $b_2$ and simultaneously
change $a_2$ to $a_2+\disc(s)+d$. It follows that
\begin{align*}
0= & \; [d]e+\disc(s)e+\{b_1\}([d]f+g)+\disc(s)\{b_1\}(f+h)\\
&+[a_2,b_1\}(f+h)+\clif(s)\{b_1\}\lambda
+\{b_2\}(g+[d]h)+\clif(s)\{b_2\}\lambda\\
&+[a_2,b_2\}(f+h)+\{b_1,b_2\}[d]\lambda
+\disc(s)\{b_1,b_2\}\lambda.
\end{align*}
Assume that $n\geq 3$, so that the invariants of $O(2n-4)$ are given
by Theorem \ref{evenorthog}. Then our knowledge of the invariants
of $O(2n-4)\times (\Z/2\times \mu_2)\times \mu_2$ from Theorems
\ref{mupprod} and \ref{zpprod}, in particular the uniqueness of the
coefficients, implies from the formula above that
$e=0$, $g=[d]f$, $h=f$, and $\lambda=0$.

So, on a quadratic form
$q=r+b_1\ll\disc(r)+d]]$, the invariant $u$ is given by:
\begin{align*}
u(q)&=c+(\clif(r)+\disc(r)\{b_1\}+[d,b_1\})f\\
&=c+\clif(q)f,
\end{align*}
for some (unique) $c\in H^{m+1,m}(k)$ and $f\in H^{m-1,m-1}(k)$.
Since the Clifford invariant is known to be an invariant of $q$,
we have determined all the invariants of $SO(q_0)$ for $n\geq 3$.

We next consider the case $n=2$. In that case, the symmetry above
(switching the two summands of a quadratic form $q=b_1\ll a_1]]+
b_2\ll a_2]]$ with $a_1+a_2=d$ over a field $F/k$)
gives only that
$[d]e=0$, $g=[d]f$, $h=f$, and $[d]\lambda=0$.
So the invariant has the form
$$u(q)=
c+[a_1]e+\clif(q)f+[a_1,b_1,b_2\}\lambda$$
for some (unique) $c\in H^{m+1,m}(k)$,
$e\in H^{m,m}(k)$, $f\in H^{m-1,m-1}(k)$,
and $\lambda\in H^{m-2,m-2}(k)$ with $[d]e=0$ and $[d]\lambda=0$.

If $b_2=b_1$, then $q=b_1(\langle\langle a_1]]+\langle\langle a_1+d]])$.
A direct calculation shows that
$$\langle\langle a_1]]+\langle\langle a_1+d]]\cong \langle\langle d]]+H$$
\cite[Example 7.23]{EKM}.
So, when $b_2=b_1$, $q$ is independent
of $a_1$, up to isomorphism. Also, when $b_2=b_1$, we have
$\{b_1,b_2\}=0$ and
$\clif(q)=[d,b_1\}$,
so $u(q)=c+[a_1]e+[d,b_1\}f$. This must be independent of $a_1$.
By the uniqueness in Theorem \ref{abgp}, 
it follows that $e=0$.
Thus the invariant $u$ has the form,
for any $a_1,a_2$ in a field $F$ over $k$
with $a_1+a_2=d$ and $b_1,b_2\in F^*$:
$$u(q)=
c+\clif(q)f+[a_1,b_1,b_2\}\lambda$$
for some (unique) $c\in H^{m+1,m}(k)$,
$f\in H^{m-1,m-1}(k)$,
and $\lambda\in H^{m-2,m-2}(k)$ with $[d]\lambda=0$.

The calculation will be finished by showing that
for any $\lambda\in H^{m-2,m-2}(k)$ with $[d]\lambda=0$,
$b_{\lambda}(q):=[a_1,b_1,b_2\}\lambda$ is an invariant for $SO(q_0)$
in $H^{m+1,m}$.
To show that $b_{\lambda}(q)$
is an invariant, we use Revoy's chain lemma
for quadratic forms in characteristic 2
\cite[Proposition 3]{Revoy}.
Write $[a,b]$ for the 2-dimensional quadratic form $ax^2+xy+by^2$.

\begin{theorem}
\label{chain}
(Revoy)
Let $k$ be a field of characteristic 2. Then the quadratic form
$\sum_{i=1}^n [a_i,b_i]$ over $k$ is isomorphic to the form
$\sum_{i=1}^n [a_i',b_i']$ if and only if these two elements of $k^{2n}$
can be connected by a sequence of the following moves:
$$A\colon [a_i,b_i]+[a_{i+1},b_{i+1}]\arrow [a_i+b_{i+1},b_i]+
[a_{i+1}+b_i,b_{i+1}]$$
for some $1\leq i\leq n-1$, or
\begin{align*}
B\colon [a_i,b_i]&\arrow [\beta^2a_i,\beta^{-2}b_i]\\
C\colon [a_i,b_i]&\arrow [a_i+\beta^2b_i+\beta,b_i]\\
D\colon [a_i,b_i]&\arrow [a_i,b_i+\beta^2a_i+\beta]
\end{align*}
for some $1\leq i\leq n$ and $\beta\in k^*$.
\end{theorem}

To relate this to the notation we have been using for quadratic
forms: an easy calculation gives that the 2-dimensional
form $[u,v]$ is isomorphic to $u\langle\langle uv]]$ if $u\neq 0$,
and to the hyperbolic plane $H=1\langle\langle 0]]$ if $u=0$.
So a 4-dimensional form $[u_1,v_1]+[u_2,v_2]$ is isomorphic
to $u_1\langle\langle u_1v_1]]+u_2\langle\langle u_2v_2]]$
if $u_1$ and $u_2$ are nonzero, with the coefficient $u_1$ changed
to 1 if $u_1=0$, and likewise for the coefficient $u_2$.
So we want to show that for any
$\lambda\in H^{m-2,m-2}(k)$ with $[d]\lambda=0$,
$$b_{\lambda}(q):=[u_1v_1,u_1,u_2\}\lambda$$
is an invariant of 4-dimensional quadratic forms
$q=[u_1,v_1]+[u_2,v_2]$ with discriminant $[d]$. (That is,
we are assuming that $u_1v_1+u_2v_2=d\in H^{1,0}(k)=k/\mathcal{P}(k)$.)
The formula for $b_{\lambda}(q)$ is understood to mean zero
if $u_1=0$ or $u_2=0$.

To show this, by Theorem \ref{chain}, it suffices to show
that $b_{\lambda}(q)$ is unchanged by moves A, B, C, or D.
One helpful observation (*) is that $[u,u\}=0$ in the Brauer group
$H^{2,1}(k)$
for all $u\in k$, where the expression is defined to mean zero
if $u=0$. This follows from the description of $H^{2,1}(k)$
in terms of differential forms (section \ref{background}),
using that $u(du/u)=du$ is exact. So we can rewrite
$b_{\lambda}(q)=[u_1v_1,u_1,u_2\}\lambda$
as $[u_1v_1,v_1,u_2\}\lambda$.
Also, we have $[u_1v_1]\lambda=[u_2v_2]\lambda$ because $[d]\lambda=0$,
and so we can also rewrite $b_{\lambda}(q)$ as $[u_2v_2,v_1,u_2\}\lambda$,
and hence as $[u_2v_2,v_1,v_2\}\lambda$, for example.

We now check that $b_{\lambda}(q)$ is unchanged by move A. After move A,
using the last formula for $b_{\lambda}(q)$ in the previous paragraph,
$b_{\lambda}(q)$ becomes
$$[(u_2+v_1)v_2,v_1,v_2\}\lambda
=[u_2v_2,v_1,v_2\}\lambda+[v_1v_2,v_1,v_2\}\lambda.$$
By relation (*), the second term is equal to $[v_1v_2,v_2,v_2\}\lambda$,
which is zero since $\{v_2,v_2\}=0$.
So the new $b_{\lambda}(q)$ is equal to the first term,
which is the old $b_{\lambda}(q)$, as we want.

Applying move B with $i=1$, the new $b_{\lambda}(q)$
is $[u_1v_1,\beta^2u_1,u_2\}\lambda=[u_1v_1,u_1,u_2\}\lambda$,
which is the old $b_{\lambda}(q)$, as we want. The same argument works
if $i=2$.

Applying move C with $i=1$, and using the last formula for $b_{\lambda}(q)$
above, the new $b_{\lambda}(q)$ is $[u_2v_2,v_1,v_2\}\lambda$, which
is the old $b_{\lambda}(q)$. Applying move C with $i=2$,
the new $b_{\lambda}(q)$
is $[u_1v_1,v_1,v_2\}\lambda$, which is equal to the old $b_{\lambda}(q)$.

Applying move D with $i=1$, the new $b_{\lambda}(q)$ is
$[u_2v_2,u_1,u_2\}\lambda$, which is the old $b_{\lambda}(q)$.
Applying move D with $i=2$,
the new $b_{\lambda}(q)$ is $[u_1v_1,u_1,u_2\}\lambda$,
which is the old $b_{\lambda}(q)$. This completes the proof that
$b_{\lambda}(q)$ is an invariant of 4-dimensional quadratic forms
with discriminant $[d]\in H^{1,0}(k)$. Thus we have found all
the invariants for $SO(q_0)$ for $q_0$ of dimension 4.

Finally, we turn to the case $n=1$. That is, given an element
$[d]\in H^{1,0}(k)$, we want to find the invariants $u$ in $H^{m+1,m}$ for
2-dimensional quadratic forms with discriminant $[d]$
over fields $F/k$.
Every such form can be written as $q=b_1\langle\langle d]]$
for some $b_1\in F^*$. The form is determined up to isomorphism
by $\{b_1\}\in H^{1,1}(F)$. So any invariant $u$ determines
an invariant for $\mu_2$ over $k$ with values in $H^{m+1,m}$.
By Proposition \ref{mup}, the invariant has the form
$$u(b_1\langle\langle d]])=c+\{b_1\}e$$
for some (unique) $c\in H^{m+1,m}(k)$ and $e\in H^{m,m-1}(k)$.
It remains to determine for which $e\in H^{m,m-1}(k)$
is $\{b_1\}e$ an invariant of $q$.

One invariant we know is the Clifford invariant of $q$,
$\clif(q)=[d,b_1\}$. It follows that any $e\in [d]H^{m-1,m-1}(k)$
gives an invariant of $q$. We show the converse.
Let $l$ be the separable
quadratic extension of $k$ with discriminant $d$. Then,
for any field $F$ over $l$ and any $b_1\in F^*$,
the form $q=b_1\langle\langle d]]$ is hyperbolic, and so $u(q)$
must be independent of $b_1$ on fields over $l$. By the uniqueness
in Proposition \ref{mup}, it follows that $e$ maps to zero
in $H^{m,m-1}(l)$. By Theorem \ref{cyclic},
$$\ker(H^{m,m-1}(k)\arrow H^{m,m-1}(l))=[d]H^{m-1,m-1}(k).$$
So $e$ is in $[d]H^{m-1,m-1}(k)$.
This completes the determination of the invariants
of $SO(q_0)$ for $q_0$ of dimension 2. Theorem \ref{oplus}
is proved.
\end{proof}

\begin{remark}
The invariant $b_{\lambda}(q)$ is easier to construct
for 4-dimensional forms $q$ with trivial discriminant,
as in the case of characteristic not 2 \cite[Example 20.3]{GMS}.
Namely, Theorem \ref{oplus} says that
$b_1(q):=[a_1,b_1,b_2\}\in H^{3,2}(F)$ is an invariant for
quadratic forms $q=b_1\langle\langle a_1]]+b_2\langle\langle
a_2]]$ over $F$ with trivial discriminant (that is, $a_1=a_2$
in $H^{1,0}(F)$).

To prove this directly, note that $q$ 
is a scalar multiple of a quadratic Pfister form,
namely $q=b_1\langle\langle b_1b_2,a_1]]$. (Following the notation
of \cite[section 9.B]{EKM}, a bilinear Pfister form
$\langle\langle a_1,\ldots,a_n\rangle\rangle$ means
$\langle\langle a_1\rangle\rangle\otimes \cdots
\otimes \langle\langle a_n\rangle\rangle$, where $\langle\langle
a\rangle\rangle$ is the 2-dimensional bilinear form
$\langle 1,-a\rangle_{\bi}$.
A quadratic Pfister form
$\langle\langle a_1,\ldots,a_n]]$ means
$\langle\langle a_1,\ldots,a_{n-1}\rangle\rangle\otimes
\langle\langle a_n]]$, where $\langle\langle a]]$ is the 2-dimensional
quadratic form $[1,a]=x^2+xy+ay^2$.)

It follows that $q$ is a difference
of two quadratic Pfister forms, $q=\langle\langle b_1,b_1b_2,a_1]]
-\langle\langle b_1b_2,a_1]]=\varphi_3-\varphi_2$,
in the quadratic Witt group
$I_q(F)$. So the class of $q$ in $I^2_q/I^3_q\cong H^{2,1}(F)$
(also known as the Clifford invariant $\clif(q)$)
is equal to the class of $\varphi_2$, and that class determines
the Pfister form $\varphi_2$ up to isomorphism,
by the Arason-Pfister Hauptsatz \cite[Theorem 23.7]{EKM}. So $q$
also determines $\varphi_3$ up to isomorphism, as $\varphi_3=q+\varphi_2$
in $I_q(F)$. The class of $\varphi_3$ in $H^{3,2}(F)$ is
$[a_1,b_1,b_1b_2\}=[a_1,b_1,b_2\}$, and so we have shown
that the latter expression is an invariant of $q$.
\end{remark}

\section{Cohomological invariants in degree 1}

In this section, we compute the mod $p$ cohomological invariants in degree 1
(that is, in $H^{1,0}$ or $H^{1,1}$)
for any affine group scheme in characteristic $p$. The analogous mod $l$
result is easier, using $A^1$-homotopy invariance of mod $l$ \etale cohomology:
for an affine group scheme $G$ over a field $k$
and a prime number $l$ invertible in $k$, the group of degree-1 invariants
for $G$ over $k$ with coefficients in $\Z/l$ is
$$H^1(k,\Z/l)\oplus \Hom_k(G,\Z/l).$$
A reference for this mod $l$ isomorphism
is Guillot \cite[Corollary 5.1.5]{Guillot}. (Guillot
assumes $k$ algebraically closed, but his proof gives this statement
for any field $k$.)

\begin{theorem}
\label{degree1}
Let $G$ be an affine group scheme of finite type over a field $k$
of characteristic $p>0$. Then
$$\Inv_k^{1,0}(G)\cong H^{1,0}(k)\oplus \Hom_k(G,\Z/p).$$
\end{theorem}

Here $H^{1,0}(k)$ can also be written as $H^1_{\et}(k,\Z/p)$.

\begin{proof}
We have $\Inv_k^{1,0}(G)\cong H^{1,0}(k)\oplus \NormInv_k(G,\Z/p)$,
as for invariants in any degree. So it suffices to identify the group
of normalized invariants with $\Hom_k(G,\Z/p)$. A homomorphism
$G\arrow \Z/p$ over $k$ clearly gives a normalized invariant
for $G$-torsors with values in $H^{1,0}(F)=H^1_{\et}(F,\Z/p)$, for
fields $F$ over $k$.

Conversely, let $\alpha$ be a normalized invariant for $G$
with values in $H^{1,0}$. Let $V$ be a representation of $G$ over $k$
with a nonempty open subset $U$ such that $G$ acts freely on $U$
with a quotient scheme $U/G$ over $k$. Applying $\alpha$
to the obvious $G$-torsor $\xi$
over the function field $k(U/G)$
determines a $\Z/p$-torsor $Y$ over $k(U/G)$. As discussed
in section \ref{background},
the invariant $\alpha$ is determined by the $\Z/p$-torsor $Y$. Since $\xi$
pulls back to a trivial $G$-torsor over $U$ and $\alpha$ is normalized,
$Y$ pulls back to a trivial $\Z/p$-torsor $Y_3$ over $U$; that is,
$Y\cong \Z/p\times \Spec\, k(U)$.

Let $G^0$ be the identity component of $G$, and let $Y_2$ be the pullback
of $Y$ over $k(U/G^0)$:
$$\xymatrix@C-10pt@R-10pt{
Y_3 \ar[r]\ar[d] & \Spec\, k(U)\ar[r]\ar[d] & U\ar[d]\\
Y_2 \ar[r]\ar[d] & \Spec\, k(U/G^0)\ar[r]\ar[d] & U/G^0\ar[d]\\
Y   \ar[r] & \Spec\, k(U/G)\ar[r] & U/G.
}$$
Since $Y_2$ is isomorphic to $Y_3/G^0$ with $G^0$ connected,
$Y_2$ also has $p$ connected components, and so the $\Z/p$-torsor
$Y_3\arrow \Spec\, k(U/G^0)$ is trivial; that is,
$Y_3\cong \Z/p\times \Spec, k(U/G^0)$.

The group scheme $G/G^0$ is finite and \etale over $k$. So $k(U/G^0)$
is a finite separable extension field of $k(U/G)$. First consider
the case where $G/G^0$ is the $k$-group scheme associated to a finite
group, which we also call $G/G^0$. Then $k(U/G^0)$ is a finite
Galois extension of $k(U/G)$ with Galois group $G/G^0$. By Galois
theory, the diagram gives a homomorphism $\alpha$
from $H:=\Gal(k(U/G)_s/k(U/G))$
to $\Z/p$ and a surjection $\beta$ from $H$ to $G/G^0$, and it shows that
the restriction of $\alpha\colon H\arrow \Z/p$ to $\ker(\beta)$
is trivial. So $\alpha$ can be identified with a homomorphism
$G/G^0\arrow \Z/p$ of finite groups. Equivalently, $\alpha
\in \NormInv_k^{1,0}(G)$ is the invariant associated to a unique
homomorphism $G\arrow \Z/p$ of $k$-group schemes, as we want.

Now consider the general case, where the finite \etale
$k$-group scheme $G/G^0$ need not be ``split'' (meaning the $k$-group
scheme associated to a finite group). Let $K$ be
the subgroup of the Galois group $H$ corresponding
to the extension $k_s(U/G)$ of $k(U/G)$, so that $H/K\cong
\Gal(k_s/k)$. Then $H^{1,0}(k(U/G))=\Hom(H,\Z/p)$ (the group
of continuous homomorphisms), which fits into an exact sequence
$$\xymatrix@C-10pt@R-10pt{
\Hom(H/K,\Z/p) \ar[r]\ar@{=}[d] & \Hom(H,\Z/p)\ar[r]\ar@{=}[d] &
  \Hom(K,\Z/p)^{H/K}\ar@{=}[d]\\
H^{1,0}(k) \ar[r] & H^{1,0}(k(U/G)) \ar[r] & H^{1,0}(k_s(U/G))^{\Gal(k_s/k)}.
}$$
The group scheme $G/G^0$ becomes split over the separable closure
$k_s$, and so the previous paragraph implies that the image of
$\alpha\in H^{1,0}(k(U/G))$ in $H^{1,0}(k_s(U/G))$ is the one
associated to a homomorphism $(G/G^0)_{k_s}\arrow \Z/p$. Since this image
is also invariant under $\Gal(k_s/k)$, it corresponds to a homomorphism
$G\arrow \Z/p$ of $k$-group schemes. Thus, letting $\alpha'$
be $\alpha$ minus
the invariant of $G$ associated to this homomorphism $G\arrow \Z/p$,
the exact sequence above shows that
$\alpha'$ is the image of an element of $H^{1,0}(k)$. Since $\alpha'$
is a normalized invariant, it follows that $\alpha'=0$. Thus we have shown
that $\alpha$ is the invariant associated to a homomorphism
$G\arrow \Z/p$ of $k$-group schemes.
\end{proof}

\begin{theorem}
Let $G$ be an affine group scheme of finite type over a field $k$
of characteristic $p>0$. Then
$$\Inv_k^{1,1}(G)\cong H^{1,1}(k)\oplus \Hom_k(G,\mu_p).$$
\end{theorem}

\begin{proof}
Let $V$ be a representation of $G$ over $k$ such that $G$ acts
freely on an open subset $U$ with a quotient scheme $U/G$
over $k$. We can assume that $V-U$ has codimension at least 2
in $V$. Let $\alpha$
be an invariant for $G$ over $k$ with values in $H^{1,1}$.
We know that $\alpha$ is determined by its class
in $H^{1,1}(k(U/G))$. Also, by Theorem \ref{bm},
this class is unramified over $U/G$; that is,
it lies in $H^0_{\Zar}(U/G,H^{1,1})$. The restriction
map $H^{1,1}(U/G)\arrow H^0(U/G,H^{1,1})$ is an isomorphism,
since both groups can be identified with the group
$H^0(U/G,\Omega^1_{\log})$ of differential forms. So we can view
$\alpha$ as a $\mu_p$-torsor over $U/G$.

Equivalently, $\alpha$ is a $G$-equivariant $\mu_p$-torsor over $U$.
We can also view this as a $G$-equivariant line bundle $L$
on $U$ with a $G$-equivariant trivialization of $L^{\otimes p}$.
Since $V-U$ has codimension at least 2 in $V$, the direct image
of $L$ from $U$ to $V$ is a line bundle. The $G$-action on $L$
and the trivialization of $L^{\otimes p}$ clearly extend to $V$.
So $\alpha$ extends uniquely to a $G$-equivariant $\mu_p$-torsor
over $V$.

The $G$-equivariant
Picard group of $V$ can be viewed as the Picard group
of the stack $[V/G]$ over $k$. By the homotopy invariance
of equivariant $K$-theory proved by Thomason, $\Pic_G(V)$ is isomorphic
to $\Pic_G(\Spec\, k)=\Hom_k(G,G_m)$ \cite[Theorem 4.1]{Thomason}.
By the exact sequence $1\arrow \mu_p\arrow G_m\arrow G_m\arrow 1$
of sheaves in the flat topology, we have an exact sequence
of flat cohomology groups over $[V/G]$:
$$(O(V)^*)^G\xrightarrow[p]{} (O(V)^*)^G\arrow H^1_G(V,\mu_p)
\arrow \Pic_G(V)\xrightarrow[p]{} \Pic_G(V).$$
Here the group of units $O(V)^*$ is equal to $k^*$, on which
$G$ acts trivially. Note that $(k^*)/(k^*)^p$ is isomorphic
to $H^1(k,\mu_p)=H^{1,1}(k)$. So this exact sequence can be rewritten
as
$$0\arrow H^{1,1}(k)\arrow H^1_G(V,\mu_p)\arrow \Hom_k(G,\mu_p)
\arrow 0.$$
Every homomorphism $G\arrow \mu_p$ determines an element of $H^1_G(V,\mu_p)$,
and so we can write
$$H^1_G(V,\mu_p)=H^{1,1}(k)\oplus \Hom_k(G,\mu_p).$$
We have an obvious homomorphism from $H^{1,1}(k)\oplus \Hom_k(G,\mu_p)$
to the group of invariants $\Inv_k^{1,1}(G)$, and this homomorphism
is an isomorphism by the description of $H^1_G(V,\mu_p)$ above.
\end{proof}


\small \sc UCLA Mathematics Department, Box 951555,
Los Angeles, CA 90095-1555

totaro@math.ucla.edu
\end{document}